\documentclass[11pt,reqno]{amsart}
\usepackage{graphicx}
\usepackage{color}
\usepackage{amsmath,amssymb,amsthm,amsfonts}
\usepackage{mathdots}
\usepackage{graphicx}
\usepackage{color}
\usepackage{multicol}
\def\R{\mathbb{R}}

\makeatletter

\newcommand{\Rmnum}[1]{\expandafter\@slowromancap\romannumeral #1@}
\makeatother

\newtheorem{thm}{Theorem}[section]

\newtheorem{lemma}[thm]{Lemma}
\newtheorem{remark}[thm]{Remark}
\newtheorem{theorem}[thm]{Theorem}

\usepackage[numbers,sort&compress]{natbib}
\allowdisplaybreaks

\begin{document}

\author{Hai-Yang Jin}
\address{School of Mathematics, South China University of Technology, Guangzhou 510640, P.R. China}
\email{mahyjin@scut.edu.cn}

\author{Zhi-An Wang\textsuperscript{$ \ast $}}\thanks{$^{\ast}$Corresponding author.}
\address{Department of Applied Mathematics, Hong Kong Polytechnic University, Hung Hom, Hong Kong,  P.R. China}
\email{mawza@polyu.edu.hk}

\author{Leyun Wu}\address{Leyun Wu
\newline\indent Department of Applied Mathematics\newline\indent Hong Kong Polytechnic University\newline\indent Hung Hom, Hong Kong,  P.R. China\newline\indent and
\newline\indent School of Mathematical Sciences, MOE-LSC,
\newline\indent Shanghai Jiao Tong University, Shanghai, P.R. China}
\email{leyunwu@126.com}

\title[Global solvability and stability of an alarm-taxis system]{Global solvability and stability of an alarm-taxis system}

\begin{abstract}
This paper is concerned with the global boundedness and stability of classical solutions to an alarm-taxis system describing the burglar alarm hypothesis as an important mechanism of anti-predation behavior when species are threaten by predators. Compared to the existing prey-taxis systems, the alarm-taxis system has more complicated coupling structure and additionally requires the gradient estimate of the primary predator density to attain the global boundedness of solutions. By the sophisticated coupling energy estimates based on the Neumann semigroup smoothing properties, we establish the existence of globally bounded solutions in two dimensions with Neumann boundary conditions and furthermore prove the global stability of co-existence homogeneous steady states under certain conditions on the system parameters.
\end{abstract}



\maketitle

\numberwithin{equation}{section}
\section{Introduction and main results}

Alarm calls are an important mechanism of anti-predation behavior when species are approached by predators, where alarm call signals may be chemical, acoustic, sound, visible movement, or any other changes that are detectable by the receiver (cf. \cite{Klump, Chiver}).  There are numerous hypotheses on the structure and function of alarm calls, among which is the ``burglar alarm'' hypothesis (cf. \cite{Burkenroad}): a prey species renders itself dangerous to a primary predator by generating an alarm call to attract a second predator at higher trophic levels in the food chain that prey on the primary predator. This attraction of a secondary predator has been observed in the marine environment where dinoflagellates bioluminesce when stimulated by disturbances from copepod feeding currents may attract a secondary predator like fish (cf. \cite{Abrahams}), and in many other species (like plants \cite{Dick}, birds \cite{Hogstedt}, primates \cite{Fichtel}). To test this hypothesis, a mathematical model was recently proposed in \cite{HB-TPB-2021}, which reads in its multi-dimensional form as
\begin{equation}\label{alarm-model}
\begin{cases}
u_{t}=d_1 \Delta u+f(u, v, w), \\
v_{t}=d_2 \Delta v-\nabla \cdot (\xi v \nabla u)+g(u, v, w),\\
w_{t}=\Delta w-\nabla \cdot (\chi w \nabla \phi(u, v))+h(u, v, w),
\end{cases}
\end{equation}
where $u, v$ and $w$  represent resource or prey (e.g. dinoflagellate),   primary predator (e.g. copepod) and secondary predator (e.g. fish), respectively; $d_1, d_2$ are positive constants representing the random dispersal rates, the positive constants $\xi$ and $\chi$ are referred to as the prey-taxis and alarm-taxis coefficients, respectively; the reaction functions $f, g, h$ describe the interspecific and/or intraspecific interactions among the prey, the primary predator and the second predator,  and $\phi(u,v)$ is a signalling function describing the intensity of the alarm signal which is produced as a result of interaction between the prey and primary predator to act as a burglar alarm attracting the secondary predator. While there are many ways one could postulate the signal intensity function $\phi(u,v)$, a simple but plausible assumption is that the signal intensity is proportional to the encounter rate between the prey and the primary predator, that is (cf. \cite{HB-TPB-2021})
$$\phi(u,v) \propto uv.$$
{\color{black}Generally the reaction functions $f,g,h$ have the prototypical forms
\begin{eqnarray}\label{reaction}
\begin{aligned}
&f(u,v,w)=\phi_1(u)-b_1 vF_1(u,v)-b_3 wF_2(u,w),\\
&g(u,v,w)=\phi_2(v)+c_1vF_1(u,v)-b_2 wF_3(v,w),\\
&h(u,v,w)=\phi_3(w)+c_2wF_3(v,w)+c_3wF_2(u,w)
\end{aligned}
\end{eqnarray}
where $b_1,b_2,c_1,c_2>0$ and $b_3,c_3\geq0$ are constants, $\phi_i\,(i=1,2,3)$ describe the intra-specific interactions of species. For the prey species, $\phi_1(u)=\mu_1u\left(1-\frac{u}{K}\right)$ where $\mu_1>0$ denotes the intrinsic growth rate and $K>0$ is the carrying capacity. For the predators $v$ an $w$, $\phi_i(s)=\mu_i s-\theta_i s^2 (i=2,3)$, where $\mu_i>0$ (reps. $<0$) denotes the intrinsic growth (reps. death) rate of species and $\theta_i\geq0$ denotes the intra-specific competition strength where $\theta_i=0$ means there is no intra-specific competition between species. In particular, if $\mu_i(i=2,3)>0$ (reps. $<0$), then the corresponding predator is called a generalist (reps. specialist) predator.   $F_i\,(i=1,2,3)$ are called the functional response (or trophic) functions describing the reproduction rate of a predator as a function of prey density, which may have many possible forms such as Holling type I, II and III, ratio-dependent or Beddington-DeAngelis type, and so on (cf. \cite{Murdoch, Turchin}). We remark that without prey-taxis and alarm-taxis, the  model \eqref{alarm-model} is generally called a food chain model (cf. \cite{Hasting-Powell}) if $b_3=c_3=0$ (i.e. the second predator $w$ does not utilize the resource), and an intraguild predation model (cf. \cite{Holt}) if $b_3,c_3>0$  (i.e. the second predator $w$ can utilize the resource of its prey $v$)}. For the alarm-taxis model, the first qualitative result was obtained in \cite{HB-TPB-2021} for the following one dimensional form with Neumann boundary conditions
\begin{equation}\label{pretaxis2}
\begin{cases}
u_t=d_1  u_{xx}+ \mu_1 u(1-u)-b_1 uv-b_3 u w, &x\in (0,L),\,\, t>0,\\
v_t=d_2 v_{xx}-(\xi vu_x)_x +\mu_2 v(1-v)+c_1 u v-b_2 v w, &x\in (0,L),\,\, t>0,\\
w_t=w_{xx}-(\chi w(vu_x+uv_x))_x +\mu_3 w(1-w)+c_2 {\color{black}vw}+c_3 uw, & x\in (0,L),\,\, t>0,\\
u_x=v_x=w_x=0, &x=0,L,\,\, t>0,\\
(u, v, w)(x,0)=(u_0, v_0, w_0)(x),&x\in (0,L),
\end{cases}
\end{equation}
where $L>0$,  $d_1, d_2, \mu_1, \mu_2, \mu_3, b_1, b_2, \xi, \chi >0$ and $b_3, c_3\geq 0$ are  constants.  The existence of global bounded solutions of \eqref{pretaxis2} was established in \cite{HB-TPB-2021}, and  the global stability of coexistence steady state for the special case $b_3=c_3=0$  was further proved under certain conditions (which will be mentioned later).

The system \eqref{alarm-model}, as the first mathematical model for alarm-taxis proposed in \cite{HB-TPB-2021}, provides basic theoretical framework to understand the  mechanism of anti-predation behavior of the prey by releasing alarm call signals.  The mathematical studies of the alarm-taxis model \eqref{alarm-model} was initiated in \cite{HB-TPB-2021} for the specialized form \eqref{pretaxis2} in one dimension only. Hence there are many interesting questions in demand to gain more insights into the understanding of the alarm-taxis mechanism, for instance the global dynamics of alarm-taxis models in a more realistic multi-dimensional spatial domain and/or with different functional response  functions $F_i$ and so on. This motivates us, among other things, to
 consider the following alarm-taxis system
\begin{equation}\label{pd-2}
\begin{cases}
u_t=d_1 \Delta u+\mu_1u(1-u)-b_1 uv-b_3 \frac{uw}{u+w},&x\in\Omega, t>0\\
v_t=d_2\Delta v-\nabla\cdot(\xi v\nabla u)+\mu_2v(1-v)+uv-b_2 vw,&x\in\Omega, t>0\\
w_t=\Delta w-\nabla \cdot[\chi w(v\nabla u+u\nabla v)]+\mu_3 w(1-w)+vw+c_3\frac{uw}{u+w},&x\in\Omega, t>0,\\
\frac{\partial u}{\partial \nu}=\frac{\partial v}{\partial \nu}=\frac{\partial w}{\partial \nu}=0, &x\in \partial\Omega,\,\, t>0,\\
(u, v, w)(x,0)=(u_0, v_0, w_0)(x),&x\in \Omega
\end{cases}
\end{equation}
in a bounded smooth domain $\Omega \subset \mathbb{R}^N (N\geq 2)$ with parameters $d_1, d_2, \mu_1, \mu_2, \mu_3, b_1, b_2, \xi, \chi >0$ and $b_3, c_3\geq 0$, where $\nu$ denotes the outward normal vector of $\partial \Omega$. Particularly when $b_3=c_3=0$ (case of food chain), the model \eqref{pd-2} is nothing but the multi-dimensional version of \eqref{pretaxis2}. The difference is that when $b_3, c_3>0$ (case of intraguild predation), the model \eqref{pd-2} employs the ratio-dependent functional response while \eqref{pretaxis2} uses the Lotka-Volterra functional response. {\color{black} The food chain model with spatial movements have not been investigated in the literature to the best of our knowledge though its ODE counterpart (i.e. the temporal model) has been extensively studied (cf. \cite{Klebanoff, McCann, Patt} and references therein). Without prey-taxis and alarm-taxis, the intraguild predation models with some particular functional response functions have been analytically studied in \cite{Cantrell-IGP1, Cantrell-IGP2}}. The main goal of this paper is to investigate the global dynamics of the alarm-taxis model \eqref{pd-2} by establishing the global boundedness of solutions in multi-dimensions and the global stability of coexistence steady states for both $b_3=c_3=0$  and $b_3, c_3>0$. To compare, we recall that the work \cite{HB-TPB-2021} obtains the global boundedness of solutions for the one dimensional model \eqref{pretaxis2}  and establish the global stability of coexistence steady states for the case  $b_3=c_3=0$ only. {\color{black} The global boundedness of solutions to \eqref{pretaxis2} in multi-dimensions still remains open and our results show that the global boundedness of classical solutions can be ensured if the interaction between $u$ and $w$ is described by the ratio-dependent functional response.}

From mathematical point of view, the structure of \eqref{alarm-model} with \eqref{reaction} is analogous to the following prey-taxis system
 \begin{equation}\label{prey}
\begin{cases}
u_t=d_1\Delta u-vF(u,v)+f(u), &x\in \Omega,\,\, t>0,\\
v_t=d_2\Delta v-\nabla \cdot ( \chi v  \nabla u)+\gamma v F(u,v)-vh(v), &x\in \Omega,\,\, t>0,\\
\frac{\partial u}{\partial \nu}=\frac{\partial v}{\partial \nu}=0, &x\in \partial\Omega,\,\, t>0,
\end{cases}
\end{equation}
where $F(u,v)$ is the functional response function which may depend on $u$ only like Holling type I, II and III or on both $u$ and $v$ like the ratio-dependent response (cf. \cite{Murdoch, Turchin}). The system \eqref{prey} is a simplified version of the original prey-taxis system proposed in \cite{Kareiva} (see also \cite{JW-EJAM-2021}), and has been studied from different analytical perspectives for different functional response functions in the literature (cf. \cite{ABN-2008, Tao-2010, HZ-2015, Li-Wang-Shao,JW-JDE-2017, WSW-JDE-2016,WW-2018-ZAMP,WX-2021-JMB}). Among other things, if $F(u,v)$ is prey-dependent only (i.e. depends on $u$ only), large-data global solutions were obtained in two dimensions \cite{JW-JDE-2017} while small-data global solutions in three or higher dimensions were attained in \cite{WSW-JDE-2016,WW-2018-ZAMP}. However if  $F(u,v)$ is ratio-dependent, large-data global solutions can be attained in any dimensions \cite{CCW-AA-2020}.  For results of other classes of taxis model in the predator-prey system  such as indirect preytaxis or predator-taxis systems, we refer to \cite{Ahn-Yoon-JDE-2020, Fuest, Wrzosek-mmas, Tell-Wrzosek-M3, Wu-Wang-Shi} and references therein. Compared to the prey-taxis system  \eqref{prey}, the alarm-taxis system \eqref{pd-2} has more intricate coupling structures where in particular a priori $\|\nabla v\|_{L^\infty}$ estimate is required to derive $\| w\|_{L^\infty}$ which however in turn affects the gradient estimates of $u$.
Hence how to untie these tangled coupling estimates to deduce the {\it a priori} estimates of $\|\nabla u\|_{L^\infty}$ and $\|\nabla v\|_{L^\infty}$ is the key to obtain the global boundedness of solutions for \eqref{pd-2}, where the estimate of $\|\nabla v\|_{L^\infty}$ is the main new challenge arising in the model to overcome. We have not found existing works addressing how to obtain the $L^\infty$-estimates for the gradient of predator densities.   In this paper, we shall first fully capture the ratio-dependent functional response structure to get the global estimate of $\|\nabla u\|_{L^\infty}$ and hence $\|v\|_{L^\infty}$. Then we start from some elegant  estimates on the second-order derivative estimate of $u$ based on works \cite{Cao-ZAMP-2016, JCH-2020-Nonlinearity} to derive the global estimate of $\|\nabla v\|_{L^\infty}$ from which the estimate of $\|w\|_{L^\infty}$ follows alongside the application of Neumann semigroup smoothing properties. We remark that a conventional method used to study the global boundedness of solutions to taxis equations (cf. \cite{JW-JDE-2016, JW-JDE-2017, Tao-Winkler-2012}) by resorting to the entropy estimates is also applicable to establish the boundedness of $u$ and $v$, but with more complicated estimates. In this paper, we provide a simpler approach by utilizing the local-in-time integrability of $L^2$-norm of $v,w$ (see Lemma \ref{Le2.3} and Lemma \ref{L2u}) resulting from the quadratic decay in the kinetics terms and by using a second-order estimate (see Lemma \ref{L4}).\\

Our first result concerning the global existence and boundedness of classical solutions of (\ref{pd-2}) is stated  in the following theorem.
\begin{theorem}[Global boundedness]\label{GB}
Let $\Omega\subset\R^2$ be a bounded domain with smooth boundary. Assume $(u_0,v_0,w_0)\in [W^{2,\infty}(\Omega)]^3$ with $u_0, v_0,w_0\gneqq 0$.  Then the problem (\ref{pd-2}) has a unique global classical solution $(u,v,w) \in [C^0(\bar{\Omega}\times[0,\infty))\cap C^{2,1}(\bar{\Omega}\times(0,\infty))]^3$ satisfying $u,v,w> 0$ for all $t>0$. Furthermore there exists a constant $C>0$ independent of $t$ such that
\begin{equation*}
\|u(\cdot,t)\|_{W^{1,\infty}}+\|v(\cdot,t)\|_{W^{1,\infty}}
+\|w(\cdot,t)\|_{L^\infty}\leq C,
\end{equation*}
where in addition $\|u\|_{L^\infty}$ is independent of $\xi$ and $\chi$ while $\|v\|_{L^\infty}$ is independent of $\chi$.
\end{theorem}

Our next results are concerned with the asymptotical behavior of solutions to the system \eqref{pd-2}. In particular we shall explore under what conditions, the positive coexistence steady state can be asymptotically achieved.  In our analysis, we just need the positivity of parameters $\mu_i\,(i=1,2,3)$ to ensure the global boundedness of solutions and the specific values of $\mu_1,\mu_2$ and $\mu_3$ are not of importance. Hence for simplicity, we assume that $\mu_1=\mu_2=\mu_3=1$ without loss of generality for the stability analysis. Then the system \eqref{pd-2} can be rewritten as
\begin{equation}\label{pd-2-1}
\begin{cases}
u_t=d_1 \Delta u+u(1-u-b_1 v-b_3 \frac{w}{u+w}),&x\in\Omega, t>0\\
v_t=d_2\Delta v-\nabla\cdot(\xi v\nabla u)+v(1-v+u-b_2 w),&x\in\Omega, t>0\\
w_t=\Delta w-\nabla \cdot[\chi w(v\nabla u+u\nabla v)]+ w(1-w+v+c_3\frac{u}{u+w}),&x\in\Omega, t>0,\\
\frac{\partial u}{\partial \nu}=\frac{\partial v}{\partial \nu}=\frac{\partial w}{\partial \nu}=0, &x\in \partial\Omega,\, t>0,\\
(u, v, w)(x,0)=(u_0, v_0, w_0)(x),&x\in \Omega.
\end{cases}
\end{equation}
Depending on whether or not the secondary predator $w$  consumes the resource $u$, we divide our analysis into two cases:
\begin{itemize}
\item[(1)] $b_3=c_3=0$: the secondary predator $w$ does not consume the (prey) resource $u$; that is the temporal dynamics is a case of food chain.
    \item[(2)] $b_3,c_3>0$: the secondary predator $w$ consumes the resource $u$; that is the temporal dynamics is a case of intraguild predation.
\end{itemize}
We first consider the case $b_3=c_3=0$, for which one can check that \eqref{pd-2-1} has three types of homogeneous (constant) steady states as follows:
\begin{itemize}
\item[1.]Trivial steady states: $(0, 0, 0),\,\, (1, 0, 0),\,\, (0, 1, 0),\,\, (0, 0, 1);$
\item[2.] Semi-trivial steady states: $(1, 0, 1),\,\, \left(\frac{1-b_1}{1+b_1}, \frac{2}{1+b_1}, 0\right),\,\, \left(0, \frac{1-b_2}{1+b_2},
\frac{2}{1+b_2}\right);$
\item[3.] Coexistence steady state: $(u^*, v^*, w^*)$ with
\begin{eqnarray}\label{uvw*}
u^*=\frac{1+b_1b_2+b_2-b_1}{1+b_1+b_2},\,\, v^*=\frac{2-b_2}{1+b_1+b_2},\,\, w^*=\frac{3+b_1}{1+b_1+b_2},
\end{eqnarray}
where $u^*, v^*, w^*$ are all positive if
\begin{equation}\label{GS-1}
    0<b_2<2\ \ \mathrm{and} \ \ 0<b_1<1+b_1b_2+b_2.
    \end{equation}
\end{itemize}
For the positive coexistence steady state $(u^*, v^*, w^*)$ defined in \eqref{GS-1}, we have the following global stability result.
\begin{theorem}[Global stability for the case of food chain]\label{GS}
Assume the assumptions in Theorem \ref{GB} hold and let $(u, v, w)$ be the solution of  \eqref{pd-2-1} with $b_3=c_3=0$. If the parameters $b_1,b_2$ satisfy \eqref{GS-1} with
\begin{equation}\label{GS-2}
(b_1-1)^2+(b_2-1)^2<4
\end{equation}
then there exist $\xi_1>0$ and $\chi_1>0$ such that whenever $\xi \in (0,\xi_1)$ and $\chi \in (0, \chi_1)$ it holds that
    \begin{equation*}\label{ec-1}
    \|u(\cdot,t)-u^*\|_{L^\infty}+\|v(\cdot,t)-v^*\|_{L^\infty}+\|w(\cdot,t)-w^*\|_{L^\infty}\leq C_1 e^{-\sigma_1 t}, \ \mathrm{for \ all}\ t>t_0,
    \end{equation*}
 with some $t_0>0$, where $C_1$ and $\sigma_1$ are positive constants independent of $t$.
\end{theorem}
\begin{figure}[h]
\centering
\includegraphics[width=10cm]{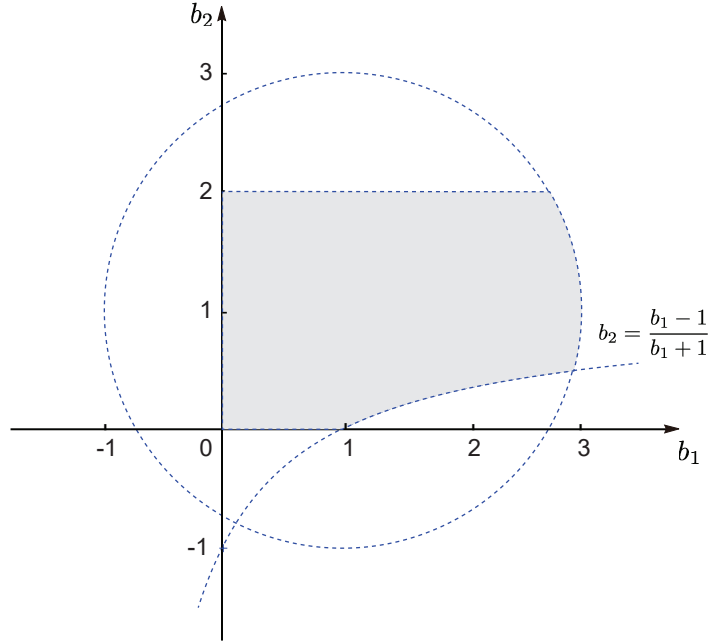}
\caption{Illustration of the admissible regime (shade region) for parameters $b_1, b_2>0$ satisfying \eqref{GS-1}-\eqref{GS-2}. }
\label{fig1}
\end{figure}
\begin{remark}
\em{
We underline that the admissible regime for the parameters $b_1, b_2>0$ satisfying \eqref{GS-1}-\eqref{GS-2} is nonempty and can be explicitly identified, see Figure \ref{fig1}.
}
\end{remark}
Next we explore the global stability of solutions in the case  $b_3, c_3>0$. For simplicity, we further assume that $c_3=1$ without loss of generality. Then
there are also three types  of homogenous steady states as follows:
\noindent
\begin{itemize}
\item[1.] Trivial steady states: $(0, 0, 1)$ and $(1, 0, 0);$

\item[2.] Semi-trivial steady states:

$\Big(0, \frac{1-b_2}{1+b_2}, \frac{2}{1+b_2}\Big),\,
\Big(\frac{1-b_1}{1+b_1}, \frac{2}{1+b_1}, 0 \Big), \ \Big(1-2b_3+\frac{2b_3^2+b_3\sqrt{2(1-b_3)}}{1+b_3}, 0, \frac{2b_3+\sqrt{2(1-b_3)}}{1+b_3}\Big);$

\item[3.] Coexistence steady state: $(u_*,v_*,w_*)$, where $u_*,v_*,w_*$ satisfy the following equations
    \begin{equation}\label{ce-1}
    \begin{cases}
    1-u_*-b_1 v_*-b_3 \frac{w_*}{u_*+w_*}=0,\\
    1-v_*+u_*-b_2 w_*=0,\\
    1-w_*+v_*+\frac{u_*}{u_*+w_*}=0.\\
    \end{cases}
    \end{equation}
 \end{itemize}
 We note that $(u_*, v_*, w_*)$ can be explicitly solved and furthermore if $b_1$ and $b_3$ are sufficiently small,  \eqref{ce-1} has a unique positive solution $(u_*,v_*,w_*)$ (see the Appendix for details), for which we have the following global stability result.

\begin{theorem}[Global stability for the case of intraguild predation]\label{GS1}
Let the assumptions in Theorem \ref{GB} hold and  $(u, v, w)$ be the solution of \eqref{pd-2-1} with $b_3>0$ and $c_3=1$. Assume that $(u_*,v_*,w_*)$ is the positive coexistence steady state satisfying \eqref{ce-1}. If
\begin{equation}\label{condition3-s1}
b_1 \mbox{ and } b_3 \mbox{ are sufficiently small, } \frac{1}{10}\leq b_2<\sqrt 2,
\end{equation}
then there exist $\xi_2>0$ and $\chi_2>0$ such that whenever $\xi \in (0,\xi_2)$ and $\chi \in (0, \chi_2)$ it holds that
\begin{equation*}\label{ce1-1}
    \|u(\cdot,t)-u_*\|_{L^\infty}+\|v(\cdot,t)-v_*\|_{L^\infty}+\|w(\cdot,t)-w_*\|_{L^\infty}\leq C_2 e^{-\sigma_2 t} \ \mathrm{for \ all}\ t>T_0,
    \end{equation*}
with some $T_0>0$,  where $C_2$ and $\sigma_2$ are positive constants independent of $t$.
 \end{theorem}

 The rest of this paper is arranged as follows. In Section 2, we  show the local existence of solutions and prove some basic properties of solutions. In Section 3, we demonstrate the details of obtaining the necessary {\it a priori} estimates of solutions and  prove Theorem \ref{GB}. Then in Section 4,  we prove the global stability of coexistence steady states under certain conditions stated in Theorem \ref{GS} and  Theorem \ref{GS1} by employing the Lyapunov functional method alongside the Barb\v{a}lat's Lemma. Section 5 is an appendix showing the existence of positive coexistence steady state under conditions imposed in Theorem \ref{GS1}.

\section{Local existence and Preliminaries}
In the sequel, we shall use  $C_i\, (i=1,2, \cdots)$ to denote a generic positive constant which may vary in the context. Without confusion, the integration variables $x$ and $t$ will be omitted, for instance $\int_0^t \int_{\Omega} f(x,s)dxds$ will be abbreviated as $\int_0^t \int_{\Omega} f(x,s)$.
The existence and uniqueness of local solutions of $\eqref{pd-2}$, which can be readily proved by  Amann's theorem \cite{Amann2, A-Book-1993}.

\begin{lemma}[Local existence]\label{LS}
	Let the assumptions in Theorem \ref{GB} hold.
Then there exists $T_{max}\in (0,\infty]$ such that the problem $\eqref{pd-2}$ has a unique classical solution $$(u,v,w) \in [C^0(\bar{\Omega}\times[0,T_{max}))\cap C^{2,1}(\bar{\Omega}\times(0,T_{max}))]^3$$ satisfying $u, v,w>0 $ for all $t>0$. Moreover,
	\begin{equation}\label{a-priori}
	\text{if}\ \ T_{max}<\infty,\ \text{then}\  \lim\limits_{t\nearrow T_{max}}(\|u(\cdot,t)\|_{W^{1,\infty}}+\|v(\cdot,t)\|_{W^{1,\infty}}
	+\|w(\cdot,t)\|_{L^\infty})=\infty .
	\end{equation}
\end{lemma}

\begin{proof}
	Denote $z=(u, v, w)$. Then the system \eqref{pd-2} can be written as
	 \begin{eqnarray}\label{LE-1}
\begin{cases}
z_{t}=\nabla \cdot (P(z) \nabla z)+Q(z), &x\in \Omega, ~~t>0,\\
\frac{\partial z}{\partial \nu}=0, &x\in \partial\Omega, ~~t>0,\\
z(\cdot, 0)=(u_0, v_0, w_0), & x\in \Omega,
\end{cases}
\end{eqnarray}
where
\begin{gather*}
P(z)=\left(\begin{matrix}
d_1 &0 &0\\
-\xi v& d_2 & 0\\
-\chi vw &-\chi u w& 1
\end{matrix}\right),\quad
Q(z)=\left(\begin{matrix}
u(\mu_1-\mu_1u-b_1v- \frac{b_3 w}{u+w})\\v(\mu_2-\mu_2v+u-b_2 w)\\w(\mu_3-\mu_3w+ v+ \frac{c_3u}{u+w})
\end{matrix}\right).
\end{gather*}
The matrix $P(z)$ is positive definite for the given initial data, which means the system \eqref{LE-1} is uniformly parabolic.
Then the application of \cite[Theorem 7.3]{Amann2} yields a $T_{max}>0$
such that the system \eqref{LE-1} possesses a unique solution $(u,v,w) \in [C^0(\bar \Omega \times [0, T_{max}))\cap C^{2,1}(\bar \Omega \times (0, T_{max}))]^3$.

Next, we prove the positivity of $u, v$ and $w.$ Applying the strong maximum principle to the first equation in \eqref{pd-2}, we have
$u(x, t) > 0$ for all $(x, t) \in \Omega \times (0, T_{max}),$   due to the fact $u_0 \gneqq 0 $. Moreover, we can rewrite the equation of $v$ as follows
\begin{equation}\label{lc-1}
  \begin{cases}
 v_{t}-d_2\Delta v+\xi\nabla u\cdot\nabla v+\Psi(x,t)=0,&x\in\Omega,t\in(0,T_{\max}),\\
 \frac{\partial v}{\partial\nu}=0,&x\in\partial\Omega,t\in(0,T_{\max}),\\
 v(x,0)=v_{0}(x)\geq0,& x\in\Omega,
 \end{cases}
\end{equation}
where $\Psi(x, t)=\xi v \Delta u -v(\mu_2-\mu_2 v+u-b_2 w).$ Then  the strong maximum principle applied to  \eqref{lc-1} yields
$v(x, t) > 0$ for all $(x, t) \in \Omega \times (0, T_{max}).$
Similarly, we can derive that $w>0$ for all $(x, t) \in \Omega \times (0, T_{max}).$ In addition, since $P(z)$ is a lower triangular matrix, the blow-up criterion  \eqref{a-priori} follows from \cite [Theorem 5.2]{Amann3} directly. Then the proof of Lemma \ref{LS} is completed.
	\end{proof}
\begin{lemma}\label{UW} Let the assumptions in Theorem \ref{GB} hold. Then
	the solution of \eqref{pd-2} satisfies
	\begin{equation}\label{Loow}
	\|u(\cdot,t)\|_{L^\infty}\leq K
	\end{equation}
for all $t>0,$ where $K:=\max\{1, \|u_0\|_{L^\infty}\},$  and
\begin{eqnarray}\label{eq2-0}
	\mathop{\lim \sup}\limits_{t\to\infty} u(\cdot, t) \leq 1 \mbox { for all } x\in \bar \Omega.
\end{eqnarray}
\end{lemma}

\begin{proof} The proof is the same as the one {\color{black} in} \cite[Lemma 2.2]{JW-JDE-2017} based on a comparison principle applied to the first equation of \eqref{pd-2} and we omit the details for brevity.
\end{proof}

\begin{lemma}\label{Le2.3}
Suppose the assumptions in Theorem \ref{GB} hold. Then the solution of \eqref{pd-2} satisfies
\begin{equation}\label{L1v}
	\|v(\cdot,t)\|_{L^1} \leq K_1\ \ \mathrm{for\ all}\ \ t\in(0,T_{max}),
	\end{equation}
and
		\begin{equation}\label{L1-2}
\int_t^{t+\tau}\int_\Omega v^2\leq \frac{4 K_1}{\mu_2}
\ \ \mathrm{for\ all}\ \ t\in(0,\widetilde{T}_{max}),
		\end{equation}
		where $K_1= \|v_0\|_{L^1}+ \frac{(\mu_2+K+1)^2|\Omega|}{2\mu_2}$,
 $\tau$ is a constant such that
		\begin{equation}\label{DTT}
		0<\tau<\min \big\{ 1,T_{max}\big\} \ \ \mathrm{and}\ \
		\widetilde{T}_{max}:=\begin{cases}
		T_{max}-\tau,&\mathrm{if}\ \ T_{max}<\infty,\\
		\infty, &\mathrm{if}\ \ T_{max}=\infty.\\
		\end{cases}
		\end{equation}
	\end{lemma}
\begin{proof}
Integrating the second equation of \eqref{pd-2} by parts with respect to $x\in\Omega$, and using the fact $\eqref{Loow}$ as well as the positivity of $(u, v, w),$ we end up with
\begin{eqnarray*}
\frac{d}{dt}\int_\Omega v+\int_\Omega v+ \mu_2 \int_\Omega v^2&=& (\mu_2+1) \int_\Omega v+\int_\Omega uv-b_2 \int_\Omega vw \\
&\leq &(\mu_2+K+1) \int_\Omega v\\
&\leq &  \frac{\mu_2}{2} \int_\Omega v^2+\frac{(\mu_2+K+1)^2|\Omega|}{2\mu_2},
\end{eqnarray*}
which gives
\begin{eqnarray}\label{ev0}
\frac{d}{dt}\int_\Omega v+\int_\Omega v+ \frac{\mu_2}{2} \int_\Omega v^2\leq \frac{(\mu_2+K+1)^2|\Omega|}{2\mu_2}.
\end{eqnarray}
Applying the Gr\"{o}nwall inequality to \eqref{ev0}, we derive
\begin{eqnarray}\label{ev1}
\int_\Omega v \leq \|v_0\|_{L^1}+ \frac{(\mu_2+K+1)^2|\Omega|}{2\mu_2}:=K_1,
\end{eqnarray}
which yields \eqref{L1v}.
Integrating \eqref{ev0} over $(t, t+\tau)$ and using \eqref{ev1}, one has
\begin{eqnarray*}\label{ev2}
\int_t^{t+\tau} \int_\Omega v^2\leq\frac{2}{\mu_2} \int_\Omega v +\frac{(\mu_2+K+1)^2|\Omega|\tau}{\mu_2^2} \leq \frac{2 \|v_0\|_{L^1}}{\mu_2}+\frac{2(\mu_2+K+1)^2|\Omega|}{\mu_2^2}\leq \frac{4K_1}{\mu_2},
\end{eqnarray*}
which gives \eqref{L1-2}.
\end{proof}

\begin{lemma}\label{L2u}
		Suppose the assumptions in Theorem \ref{GB} hold. Then  there exist two constants $K_2>0$ and $K_3>0$ which are independent of $\xi$ and $\chi$ such that the solution of \eqref{pd-2} satisfies
\begin{equation}\label{L1w}
	\|w(\cdot,t)\|_{L^1}\leq K_2\ \ \mathrm{for\ all}\ \ t\in(0,T_{max}),
	\end{equation}
and
		\begin{equation}\label{L1w-t}
\int_t^{t+\tau}\int_\Omega w^2\leq K_3
\ \ \mathrm{for\ all}\ \ t\in(0,\widetilde{T}_{max}),
		\end{equation}
		where $\tau$ and $  \widetilde{T}_{max}$ are defined by \eqref{DTT}.
	\end{lemma}

\begin{proof}
Multiplying  the third equation in \eqref{pd-2} by $b_2$,  adding the result to the second equation in \eqref{pd-2}, followed by an integration,  we have
\begin{equation}\label{L1w-1}
\begin{split}
&\frac{d}{dt} \int_\Omega (v+b_2w)+\mu_2\int_\Omega v^2+b_2\mu_3\int_\Omega w^2\\
&=\mu_2\int_\Omega v+\int_\Omega uv +b_2\mu_3\int_\Omega w+b_2c_3\int_\Omega\frac{uw}{u+w}.\\
\end{split}
\end{equation}
Adding $\int_\Omega (v+b_2w)$ to both sides of \eqref{L1w-1}, along with \eqref{Loow} and \eqref{L1v},  one obtains
\begin{equation*}
\begin{split}
&\frac{d}{dt} \int_\Omega (v+b_2w)+ \int_\Omega (v+b_2w)+\mu_2\int_\Omega v^2+b_2\mu_3\int_\Omega w^2\\
&=(\mu_2+1)\int_\Omega v+\int_\Omega uv +b_2(\mu_3+1)\int_\Omega w+b_2c_3\int_\Omega\frac{uw}{u+w}\\
&\leq (\mu_2+1+K)\int_\Omega v+b_2(\mu_3+1+c_3)\int_\Omega w\\
&\leq (\mu_2+1+K)K_1+b_2(\mu_3+1+c_3)\int_\Omega w\\
&\leq (\mu_2+1+K)K_1+\frac{b_2\mu_3}{2}\int_\Omega w^2+\frac{b_2(\mu_3+1+c_3)^2|\Omega|}{2\mu_3},
\end{split}
\end{equation*}
which gives
\begin{equation}\label{L1w-2}
\frac{d}{dt} \int_\Omega (v+b_2w)+ \int_\Omega (v+b_2w)+\frac{b_2\mu_3}{2}\int_\Omega w^2\leq (\mu_2+1+K)K_1+\frac{b_2(\mu_3+1+c_3)^2|\Omega|}{2\mu_3}.
\end{equation}
With Gr\"{o}nwall's inequality applied to \eqref{L1w-2}, one has
\begin{equation}\label{L1w-3}
\int_\Omega (v+b_2w)\leq  \int_\Omega (v_0+b_2w_0)+(\mu_2+1+K)K_1+\frac{b_2(\mu_3+1+c_3)^2|\Omega|}{2\mu_3},
\end{equation}
which gives \eqref{L1w} by defining
\begin{equation}\label{K2D}
K_2:=\frac{\|v_0\|_{L^1}+b_2\|w_0\|_{L^1}+ (\mu_2+1+K)K_1+\frac{b_2(\mu_3+1+c_3)^2|\Omega|}{2\mu_3}}{b_2}.
\end{equation}
Then integrating \eqref{L1w-2} over $(t, t+\tau)$, and using \eqref{L1w-3} and \eqref{K2D}, we derive that
\begin{equation*}
\begin{split}
\frac{b_2\mu_3}{2}\int_t^{t+\tau}\int_\Omega  w^2
&\leq \int_\Omega (v+b_2w)+\left[(\mu_2+1+K)K_1+\frac{b_2(\mu_3+1+c_3)^2|\Omega|}{2\mu_3}\right]\tau\\
&\leq b_2K_2+K_1+(\mu_2+1+K)K_1+\frac{b_2(\mu_3+1+c_3)^2|\Omega|}{2\mu_3},
\end{split}
\end{equation*}
which gives \eqref{L1w-t} by letting
\begin{equation}\label{K3}
K_3:=\frac{2}{b_2\mu_3}\cdot\left( b_2K_2+(\mu_2+2+K)K_1+\frac{b_2(\mu_3+1+c_3)^2|\Omega|}{2\mu_3}\right).
\end{equation}
The proof of Lemma \ref{L2u} is completed.
\end{proof}


For later use, we list some well-known $L^p$-$L^q$ estimates for the Neumann heat semigroup (cf. \cite{Winkler-2010-JDE}).
\begin{lemma} {\color{black}\label{heatsemigroup} Let $(e^{td\Delta})_{t\geq 0}$ be the Neumann heat semigroup in $\Omega$,  and let $\lambda_1>0$ denote the first nonzero eigenvalue of $-\Delta$ in $\Omega$ under Neumann boundary conditions, where $d$ is a positive constant.} Then for all $t>0$, there exist some constants  $\gamma_i\,(i=1,2,3,4)$ depending only on $\Omega$ such that

$(\mathrm{i})$ ~~If $2\leq p<\infty$, then
\begin{equation}\label{Lp-1}
\|\nabla e^{td\Delta}z\|_{L^p}\leq \gamma_1 (1+t^{-\frac{n}{2}(\frac{1}{q}-\frac{1}{p})})e^{-d\lambda_1t}\|\nabla z\|_{L^q}
\end{equation}
for all $z\in W^{1,q}(\Omega)$.

$(\mathrm{ii})$ ~~ If  $1\leq q\leq p\leq\infty$, then
\begin{equation}\label{Lp-2}
\|\nabla e^{td\Delta}z\|_{L^p}\leq \gamma_2\left(1+t^{-\frac{1}{2}-\frac{n}{2}(\frac{1}{q}-\frac{1}{p})}\right)e^{-d\lambda_1 t}\|z\|_{L^q}
\end{equation}
for all $z\in L^q(\Omega)$.

$(\mathrm{iii})$ ~~If $1\leq q\leq p\leq \infty$, then
\begin{equation}\label{Lp-3}
\|e^{td\Delta}z\|_{L^p}\leq \gamma_3\left(1+t^{-\frac{n}{2}(\frac{1}{q}-\frac{1}{p})}\right) \|z\|_{L^q}
\end{equation}
for all $z\in L^q(\Omega)$.

$(\mathrm{iv})$ ~~If $1< q\leq p\leq \infty$, then
\begin{equation}\label{Lp-4}
\|e^{td\Delta}\nabla\cdot z\|_{L^p}\leq \gamma_4\left(1+t^{-\frac{1}{2}-\frac{n}{2}(\frac{1}{q}-\frac{1}{p})}\right)e^{-d\lambda_1 t}\|z\|_{L^q}
\end{equation}
 for all $z\in (C_0^\infty(\Omega))^n$.
\end{lemma}

\section{Proof of Theorem \ref{GB}}
In this section, we shall derive the global {\it a priori} estimates of solutions to \eqref{pd-2} which enable us to extend local solutions to global ones. Since the global stability analysis in the next section requires us to elucidate how the upper bounds of $\|v\|_{L^\infty}$ and $\|w\|_{L^\infty}$ depend on the system parameters, we shall keep the dependencies of relevant estimates on the system parameters for later use although estimates uniformly in time will suffice to obtain the global existence/boundedness of solutions.
\subsection{Boundedness of $\|v(\cdot,t)\|_{L^\infty}$} 
Since the conditions imposed for the global stability of coexistence steady states shown in section 4 depend on $\|v(\cdot,t)\|_{L^\infty}$, we shall detail the dependencies of the upper bound of $\|v(\cdot,t)\|_{L^\infty}$ on the system parameters as transparent as possible for later use.
\begin{lemma}\label{L2}
Let the assumptions in Theorem \ref{GB} hold and $(u, v, w)$ be the solution of   \eqref{pd-2}. Then there exist two positive constants $K_4$ and $K_5$ which are  independent of $\xi$ and $\chi$ such that
\begin{equation}\label{L2u-1}
\|\nabla u\|_{L^2}\leq K_4 \,\,\mathrm{\ for \ all}\ t\in (0, {T}_{max}),
\end{equation}
and
\begin{equation}\label{L2u-2}
\int_t^{t+\tau}\int_\Omega |\Delta u|^2 \leq K_5\,\, \mathrm{\ for \ all}\ t\in (0, \widetilde{T}_{max}),
\end{equation}
where
 $\tau$ and $\widetilde{T}_{max}$ are defined in \eqref{DTT}.
\end{lemma}
\begin{proof}
Multiplying the first equation of \eqref{pd-2}  by $-\Delta u$ and using the fact that $\|u(\cdot,t)\|_{L^\infty}\leq K$ in Lemma \ref{UW}, we end up with
\begin{equation*}
\begin{split}
\frac{1}{2}\frac{d}{dt}\int_\Omega |\nabla u|^2	+d_1 \int_\Omega |\Delta u|^2
&=-\mu_1\int_\Omega u(1-u)\Delta u +b_1\int_\Omega uv\Delta u +b_3\int_\Omega \frac{uw}{u+w}\Delta u \\
&\leq  K[\mu_1(1+K)+b_3]\int_\Omega |\Delta u|+b_1K\int_\Omega v|\Delta u|\\
&\leq \frac{d_1}{2}\int_\Omega |\Delta u|^2+\frac{K^2[\mu_1(1+K)+b_3]^2|\Omega|}{d_1}+\frac{b_1^2K^2}{d_1}\int_\Omega v^2,
\end{split}
\end{equation*}
which gives
\begin{eqnarray}\label{L2u-3}
\frac{d}{dt}\int_\Omega |\nabla u|^2+d_1 \int_\Omega |\Delta u|^2
\leq \frac{2b_1^2K^2}{d_1}\int_\Omega v^2+\frac{2K^2[\mu_1(1+K)+b_3]^2|\Omega|}{d_1}.
\end{eqnarray}
On the other hand, using \eqref{Loow} and  Young's inequality, we have
\begin{equation*}
\int_\Omega |\nabla  u|^2=\int_\Omega \nabla u\cdot \nabla u=-\int_\Omega u\Delta u\leq K\int_\Omega|\Delta u|\leq \frac{d_1}{4} \int_\Omega |\Delta  u|^2+\frac{K^2}{d_1}|\Omega|,
\end{equation*}
which, substituted into \eqref{L2u-3}, gives
\begin{equation}\label{L2u-4}
\begin{split}
\frac{d}{dt}\int_\Omega |\nabla u|^2+\int_\Omega |\nabla u|^2+\frac{3d_1}{4} \int_\Omega |\Delta u|^2
\leq \frac{2b_1^2K^2}{d_1}\int_\Omega v^2+\frac{K^2\left(2[\mu_1(1+K)+b_3]^2|\Omega|+|\Omega|\right)}{d_1}.
\end{split}
\end{equation}
Then applying the Gronwall type inequality (e.g. see \cite[Lemma 3.4]{SSW14}) and using the fact \eqref{L1-2}, from \eqref{L2u-4}, we obtain \eqref{L2u-1} with
\begin{equation}\label{K4}
K_4:=\left[\|\nabla u_0\|_{L^2}^2+\frac{2\tau+1}{\tau}\cdot \left(\frac{4b_1^2K^2 K_1}{d_1\mu_2}+\frac{K^2\left(2[\mu_1(1+K)+b_3]^2|\Omega|+|\Omega|\right)}{d_1}\right)\right]^\frac{1}{2}.
\end{equation}
Now we integrate \eqref{L2u-4} over $(t, t+\tau)$, and use \eqref{L1-2} and \eqref{L2u-1} to obtain  for all $t\in (0, \widetilde{T}_{max})$ that
\begin{equation*}
\begin{split}
\frac{3d_1}{4}\int_t^{t+\tau}\int_\Omega |\Delta u|^2
&\leq \int_\Omega |\nabla u|^2+\frac{2b_1^2K^2}{d_1}\int_t^{t+\tau}\int_\Omega v^2 +\frac{K^2\left(2[\mu_1(1+K)+b_3]^2|\Omega|+|\Omega|\right)}{d_1}\\
&\leq K_4^2+\frac{8b_1^2K^2 K_1}{d_1\mu_2}+\frac{K^2\left(2[\mu_1(1+K)+b_3]^2|\Omega|+|\Omega|\right)}{d_1},
\end{split}
\end{equation*}
which gives \eqref{L2u-2} with
\begin{equation}\label{K5}
K_5:=\frac{4}{3d_1}\cdot\left(K_4^2+\frac{8b_1^2K^2 K_1}{d_1\mu_2}+\frac{K^2\left(2[\mu_1(1+K)+b_3]^2|\Omega|+|\Omega|\right)}{d_1}\right).
\end{equation}
Then we complete the proof of Lemma \ref{L2}.
\end{proof}
\begin{lemma}\label{L3}
Let the assumptions in Theorem \ref{GB} hold and $(u, v, w)$ be the solution of \eqref{pd-2}. Then there exists a positive constant $K_6$ independent of $\chi$ such that
\begin{equation}\label{L3-1}
\|v(\cdot,t)\|_{L^2}\leq K_6\,\,\mathrm{\ for \ all}\ t\in (0, {T}_{max}).
\end{equation}
\end{lemma}
\begin{proof}
Multiplying the second equation of the system \eqref{pd-2} by $v$, integrating the result by parts, and appealing to Young's inequality, we obtain
\begin{equation*}
\begin{split}
&\frac{1}{2}\frac{d}{dt}\int_\Omega v^2 +	 d_2 \int_\Omega |\nabla v|^2+\mu_2\int_\Omega v^3+b_2\int_\Omega v^2 w\\
&=\xi \int_\Omega v \nabla u \cdot\nabla v +\mu_2\int_\Omega v^2+\int_\Omega u v^2\\
&\leq \frac{d_2}{2}\int_\Omega |\nabla v|^2+\frac{\xi^2}{2d_2}\int_\Omega v^2 |\nabla u|^2 +( K +\mu_2)\int_\Omega v^2\\
&\leq  \frac{d_2}{2}\int_\Omega |\nabla v|^2+\frac{\xi^2}{2d_2}\left(\int_\Omega v^4\right)^{\frac{1}{2}} \left(\int_\Omega|\nabla u|^4\right)^{\frac{1}{2}}+( K +\mu_2)\int_\Omega v^2,
\end{split}
\end{equation*}
which gives
\begin{eqnarray}\label{eq2-14}
\frac{d}{dt}\int_\Omega v^2 +d_2 \int_\Omega |\nabla v|^2\leq \frac{\xi^2}{d_2}\left(\int_\Omega v^4\right)^{\frac{1}{2}} \left(\int_\Omega|\nabla u|^4\right)^{\frac{1}{2}}+2( K +\mu_2)\int_\Omega v^2.
\end{eqnarray}
Using the Gagliardo-Nirenberg inequality in two dimensions, we can find two positive constants $C_1$ and $C_2$  such that
 \begin{eqnarray}\label{eq2-15}
 \|v\|_{L^4}^2\leq C_1(\|\nabla v\|_{L^2}\|v\|_{L^2}+\|v\|_{L^2}^2),	
 \end{eqnarray}
and
 \begin{eqnarray}\label{eq2-16}
 	\|\nabla u\|_{L^4}^2 \leq C_2 (\|\Delta u\|_{L^2}\|\nabla u\|_{L^2}+\|\nabla u\|_{L^2}^2),
 \end{eqnarray}
 where the boundary condition $\frac{\partial u}{\partial \nu}|_{\partial \Omega}=0$ has been used to obtain \eqref{eq2-16} (see \cite[Lemma 2.5]{JKW-SIAP-2018}).
Then using the fact $\|\nabla u\|_{L^2}\leq K_4$ (see Lemma \ref{L2}), \eqref{eq2-15} and \eqref{eq2-16}, one has
 \begin{equation}\label{eq2-17}
 \begin{split}
 &\frac{\xi^2}{d_2}\left(\int_\Omega v^4\right)^{\frac{1}{2}} \left(\int_\Omega|\nabla u|^4\right)^{\frac{1}{2}}\\
 &\leq \frac{\xi^2C_1C_2}{d_2}(\|\nabla v\|_{L^2}\|v\|_{L^2}+\|v\|_{L^2}^2)(\|\Delta u\|_{L^2}\|\nabla u\|_{L^2}+\|\nabla u\|_{L^2}^2)\\
 &\leq \frac{\xi^2 C_1C_2K_4}{d_2} \|\nabla v\|_{L^2}\|v\|_{L^2}\|\Delta u\|_{L^2}+\frac{\xi^2 C_1C_2K_4^2}{d_2}\|\nabla v\|_{L^2}\|\|v\|_{L^2}\\
 &\ \ \ \ +\frac{\xi^2 C_1C_2K_4}{d_2}\|v\|_{L^2}^2\|\Delta u\|_{L^2}+\frac{\xi^2 C_1C_2K_4^2}{d_2}\|v\|_{L^2}^2 \\
 &\leq  d_2  \|\nabla v\|^2_{L^2}+C_3\|v\|_{L^2}^2\|\Delta u\|_{L^2}^2+C_4\|v\|_{L^2}^2
 \end{split}
\end{equation}
with $C_3=\frac{(\xi^2C_1C_2K_4)^2(1+d_2^2)}{2d_2^3}$ and $C_4=\frac{(\xi^2C_1C_2K_4^2)^2+d_2^2+2d_2^2
\xi^2C_1C_2K_4^2}{2d_2^3}$. Substituting \eqref{eq2-17} into \eqref{eq2-14} and using Young's inequality, we derive for all $t\in (0, T_{max})$ that
\begin{eqnarray}\label{eq2-18}
\frac{d}{dt}\|v\|_{L^2}^2\leq C_3\|v\|_{L^2}^2\|\Delta u\|	_{L^2}^2+\left[C_4+2( K +\mu_2)\right]\|v\|_{L^2}^2.
\end{eqnarray}
 By virtue of \eqref{L1-2}, there exits a constant $t_0 \in
 [(t-\tau)_+, t)$ for any $t\in (0, T_{max})$ such that
 \begin{eqnarray}\label{eq2-19}
 \|v(\cdot, t_0)\|_{L^2}^2 \leq C_5:= \max\left\{\|v_0\|_{L^2}^2, \frac{4K_1}{\mu_2\tau}\right\}
 \end{eqnarray}
in both cases $t \in (0, \tau)$ and $t \geq \tau,$ where $\tau$ is defined in \eqref{DTT}. We can also derive from  \eqref{L2u-2} that
 \begin{eqnarray}\label{eq2-20}
 	\int_{t_0}^{t_0+\tau}\int_\Omega |\Delta u(\cdot, s)|^2\leq K_5.
 \end{eqnarray}
Noticing the fact $t_0<t\leq t_0+\tau \leq t_0+1,$ we integrate \eqref{eq2-18} over $(t_0, t)$ alongside \eqref{eq2-19}-\eqref{eq2-20} and get
\begin{equation*}
 	\|v(\cdot, t)\|_{L^2}^2
 \leq \|v(\cdot, t_0)\|_{L^2}^2e^{C_3\int_{t_0}^{t}\|\Delta u(\cdot, s)\|_{L^2}^2ds+C_4+2(K+\mu_2)}
 \leq C_5 e^{C_3 K_5+C_4+2(K+\mu_2)}
 \end{equation*}
 for all $t \in (0, T_{max}),$ which gives \eqref{L3-1} by defining
 \begin{equation}\label{K6}
 K_6:=C_5^\frac{1}{2} e^\frac{C_3 K_5+C_4+2(K+\mu_2)}{2}.
 \end{equation}
 This completes the proof.
\end{proof}
\begin{lemma}\label{Le3.3}
Let the assumptions in Theorem \ref{GB} hold and $(u, v, w)$ be the solution of the system  \eqref{pd-2}.  Then there exists a positive constant $K_7$ independent of $\chi$ such that
\begin{equation}\label{LIV*-1}
\|\nabla u(\cdot,t)\|_{L^4}\leq K_7\,\,\mathrm{\ for \ all}\ t\in (0, {T}_{max}).
\end{equation}
\end{lemma}
\begin{proof}
To begin with, we rewrite the first equation of \eqref{pd-2} as follows:
\begin{equation}\label{LIV-2}
u_t-d_1\Delta u+u=F(x,t)
\end{equation}
with
$F(x,t):=\mu_1u(1-u)-b_1 uv-b_3\frac{uw}{u+w}+u.$ Then by \eqref{Loow} and \eqref{L3-1}, for all $t\in(0,T_{max})$, one has
\begin{equation}\label{F2}
\begin{split}
\|F(\cdot,t)\|_{L^2}
&= \big\|\mu_1u(1-u)-b_1 uv-b_3\frac{uw}{u+w}+ u\big\|_{L^2}\\
&\leq \mu_1 K(1+K)|\Omega|^\frac{1}{2}+b_1K\|v\|_{L^2}+b_3K|\Omega|^\frac{1}{2}+K|\Omega|^\frac{1}{2}\\
&\leq [\mu_1(1+K)+b_3+1]K|\Omega|^\frac{1}{2}+b_1KK_6.
\end{split}
\end{equation}
We  apply the variation-of-constants formula to \eqref{LIV-2} and obtain
\begin{equation}\label{vq-1}
u(\cdot,t)=e^{(d_1\Delta -1)t}u_0+\int_0^t
e^{(d_1\Delta-1)(t-s)}F(\cdot, s)ds.
\end{equation}
Then using the  estimates \eqref{Lp-1} and \eqref{Lp-2}, one can derive from \eqref{vq-1} that
\begin{equation*}
\begin{split}
\|\nabla u(\cdot,t)\|_{L^4}
&\leq \|\nabla e^{(d_1\Delta -1)t}u_0 \|_{L^4}+\int_0^{t}\|\nabla e^{(d_1\Delta-1)(t-s)}F(\cdot,s)\|_{L^4}ds\\
&\leq e^{-t}\|\nabla e^{t d_1\Delta }u_0\|_{L^4}+\int_0^t\|\nabla e^{(t-s)d_1\Delta} F(\cdot,s)\|_{L^4}ds\\
&\leq 2 \gamma_1e^{-d_1\lambda_1t}\|\nabla u_0\|_{L^4}
+\gamma_2\int_0^t
\left(1+(t-s)^{-\frac{3}{4}}\right)
e^{-d_1\lambda_1(t-s)}\|F(\cdot,s)\|_{L^2}ds\\
&\leq 2 \gamma_1\|\nabla u_0\|_{L^4}+{\gamma_2} \left( [\mu_1(1+K)+b_3+1]K|\Omega|^\frac{1}{2}+b_1KK_6\right)\int_0^{\infty}\left(1+z^{-\frac{3}{4}}\right)
e^{-d_1\lambda_1z}dz\\
&\leq 2\gamma_1\|\nabla u_0\|_{L^4}+\frac{\gamma_{2}}{d_1\lambda_1}\left( [\mu_1(1+K)+b_3+1]K|\Omega|^\frac{1}{2}+b_1KK_6\right)\left(1+\Gamma(1/4)(d_1\lambda_{1})^{\frac{3}{4}}\right),
\end{split}
\end{equation*}
where $\Gamma$ denotes the usual Gamma function defined by
$
\Gamma(z)=\int_0^\infty t^{z-1}e^{-t}dt.
$
This  gives \eqref{LIV*-1} with
\begin{equation}\label{K7}
K_7:=2\gamma_1\|\nabla u_0\|_{L^4}+\frac{\gamma_{2}}{d_1\lambda_1}\left( [\mu_1(1+K)+b_3+1]K|\Omega|^\frac{1}{2}+b_1KK_6\right)\left(1+\Gamma(1/4)(d_1\lambda_{1})^{\frac{3}{4}}\right),
\end{equation}
and completes the proof of this lemma.
\end{proof}
 \begin{lemma}\label{LIU}
Let the assumptions in Theorem \ref{GB} hold and $(u, v, w)$ be the solution of  \eqref{pd-2}.  Then there exists a positive constant $K_8$ independent of $\chi$ such that
\begin{equation}\label{LIU-1}
\| v(\cdot,t)\|_{L^3}\leq K_8\mathrm{\ for \ all}\ t\in (0, {T}_{max}).
\end{equation}
\end{lemma}

\begin{proof}
Multiplying the second equation of  \eqref{pd-2} by
$v^2$ and integrating the result over $\Omega$ along with \eqref{Loow} and \eqref{LIV*-1}, we get
\begin{equation*}
\begin{split}
&\frac{1}{3}\frac{d}{dt}\int_\Omega v^3+2d_2\int_\Omega v |\nabla v|^2+\mu_2\int_\Omega v^4 \\
&=2\xi \int_\Omega v^2 \nabla u\cdot \nabla v+\mu_2\int_\Omega v^3+\int_\Omega uv^3 -b_2\int_\Omega  v^3 w\\
&\leq  d_2 \int_\Omega v |\nabla v|^2+\frac{\xi^2}{d_2} \int_\Omega v^3 |\nabla u|^2+(\mu_2+K)\int_\Omega v^3\\
&\leq  d_2 \int_\Omega v |\nabla v|^2+\frac{\xi^2}{d_2} \left(\int_\Omega v^6\right)^{\frac{1}{2}} \left(\int_\Omega|\nabla u|^4\right)^{\frac{1}{2}}+(\mu_2+K)\int_\Omega v^3\\
&\leq d_2 \int_\Omega v |\nabla v|^2+\frac{K_7^2\xi^2}{d_2}\|v\|_{L^6}^3+(\mu_2+K)\|v\|_{L^3}^3,
\end{split}
\end{equation*}
which implies
\begin{equation}\label{eq2-23}
\begin{split}
&\frac{d}{dt}\|v\|_{L^3}^3+\frac{4d_2}{3}\|\nabla v^\frac{3}{2}\|_{L^2}^2+3\mu_2\|v\|_{L^4}^4 \leq \frac{3K_7^2\xi^2}{d_2}\|v\|_{L^6}^3+3(\mu_2+K)\|v\|_{L^3}^3.
\end{split}
\end{equation}
From \eqref{L3-1}, it follows  that $\|v^\frac{3}{2}(\cdot,t)\|_{L^\frac{4}{3}}=\|v(\cdot,t)\|_{L^2}^\frac{3}{2}\leq K_6^\frac{3}{2}$. Then applying the Gagliardo-Nirenberg inequality and Young's inequality, we can derive that
\begin{eqnarray}\label{eq2-24}
\begin{split}
\frac{3K_7^2\xi^2}{d_2}\|v\|_{L^6}^3	
&=  \frac{3K_7^2\xi^2}{d_2} \|v^{\frac{3}{2}}\|^2_{L^4}\\
&\leq  \frac{3K_7^2\xi^2C_1}{d_2} \left(\|\nabla v^{\frac{3}{2}}\|_{L^2}^{\frac{4}{3}}
\|v^{\frac{3}{2}}\|_{L^\frac{4}{3}}^{\frac{2}{3}}+\|v^{\frac{3}{2}}\|_{L^\frac{4}{3}}^2\right)\\
&\leq  \frac{3K_7^2\xi^2C_1K_6}{d_2} \|\nabla v^{\frac{3}{2}}\|_{L^2}^{\frac{4}{3}}
+\frac{3K_7^2\xi^2C_1K_6^3}{d_2}\\
&\leq \frac{4d_2}{3}\|\nabla v^{\frac{3}{2}}\|_{L^2}^{2}+C_{2},
\end{split}
\end{eqnarray}
 where $C_{2}=\frac{9(K_7^2\xi^2C_1K_6)^3}{4d_2^5}+\frac{3K_7^2\xi^2C_1K_6^3}{d_2}$.
 On the other hand, using  Young's inequality, one has
 \begin{equation}\label{eq2-25}
 [3(\mu_2+K)+1]\|v\|_{L^3}^3\leq 3\mu_2\|v\|_{L^4}^4+C_3,
 \end{equation}
 where $C_3=4^{-4}\mu_2^{-3}[3(\mu_2+K)+1]^4|\Omega|$. Then  adding $\|v\|_{L^3}^3$ on both sides of \eqref{eq2-23} and  substituting \eqref{eq2-24}-\eqref{eq2-25} into
the resulting inequality, we obtain
 $$
\frac{d}{dt}\|v\|_{L^3}^3+\|v\|_{L^3}^3\leq C_2+C_3,
$$
which immediately  gives \eqref{LIU-1}
with
$
	K_8:=\left(\|v_0\|_{L^3}^3+C_2+C_3\right)^\frac{1}{3}.
$
This completes the proof of Lemma \ref{LIU}.
\end{proof}

\begin{lemma}\label{LIV}
Let the assumptions in Theorem \ref{GB} hold and $(u, v, w)$ be the solution of system  \eqref{pd-2}.  Then there exists a positive constant $K_9$ independent of $\chi$ such that
\begin{equation}\label{LIV-1}
\|v(\cdot,t)\|_{L^\infty}\leq K_9\,\,\mathrm{\ for \ all}\ t\in (0, {T}_{max}).
\end{equation}
\end{lemma}

\begin{proof}
Using Lemma \ref{LIU} and the fact $0<u\leq K$ (see Lemma \ref{UW}), one has
\begin{equation}\label{F3}
\begin{split}
\|F(\cdot,t)\|_{L^3}
&= \big\|\mu_1u(1-u)-b_1 uv-b_3\frac{uw}{u+w}+ u\big\|_{L^3}\\
&\leq \mu_1 K(1+K)|\Omega|^\frac{1}{3}+b_1K\|v\|_{L^3}+b_3K|\Omega|^\frac{1}{3}+K|\Omega|^\frac{1}{3}\\
&\leq [\mu_1(1+K)+b_3+1]K|\Omega|^\frac{1}{3}+b_1KK_8.
\end{split}
\end{equation}
Taking $\nabla$ on both sides of \eqref{vq-1} and using  $L^p$-$L^q$ estimates \eqref{Lp-2}, we have from \eqref{F3} that
\begin{equation*}
\begin{split}
\|\nabla u(\cdot,t)\|_{L^\infty}
&\leq \|\nabla e^{(d_1\Delta -1)t}u_0 \|_{L^\infty}+\int_0^{t}\|\nabla e^{(d_1\Delta-1)(t-s)}F(\cdot,s)\|_{L^\infty}ds\\
&\leq c_1\| u_0\|_{W^{1,\infty}}+\int_0^{t}\|\nabla e^{(t-s)d_1\Delta} F(\cdot,s)\|_{L^\infty}ds\\
&\leq c_1\| u_0\|_{W^{1,\infty}}
+\gamma_2\int_0^{t}
\left(1+(t-s)^{-\frac{5}{6}}\right)
e^{-d_1\lambda_1(t-s)}\|F(\cdot,s)\|_{L^3}ds\\
&\leq c_1\| u_0\|_{W^{1,\infty}}+{\gamma_2}\left([\mu_1(1+K)+b_3+1]K|\Omega|^\frac{1}{3}+b_1KK_8\right) \int_0^{\infty}\left(1+z^{-\frac{5}{6}}\right)
e^{-d_1\lambda_1z}dz\\
&\leq c_1\| u_0\|_{W^{1,\infty}}+\frac{\gamma_{2}}{d_1\lambda_1}\left( [\mu_1(1+K)+b_3+1]K|\Omega|^\frac{1}{3}+b_1KK_8\right)\left(1+(d_1\lambda_{1})^{\frac{5}{6}}\Gamma(1/6)\right),
\end{split}
\end{equation*}
which gives
\begin{equation}\label{LIU*}
\|\nabla u(\cdot,t)\|_{L^\infty}\leq K_8^*
\end{equation}
with
\begin{equation*}\label{K8}
K_8^*:=c_1\| u_0\|_{W^{1,\infty}}+\frac{\gamma_{2}}{d_1\lambda_1}\left([\mu_1(1+K)+b_3+1]K|\Omega|^\frac{1}{3}+b_1KK_8\right)\left(1+(d_1\lambda_{1})^{\frac{5}{6}}\Gamma(1/6)\right).
\end{equation*}
We rewrite the second equation of \eqref{pd-2} as follows:
\begin{equation}\label{eq3.30}
v_t-d_2\Delta v+v=-\xi\nabla\cdot(v\nabla u)+\mu_2v(1-v)+uv-b_2 vw+v.
\end{equation}
Then   applying the variation-of-constants formula to \eqref{eq3.30},  one has
\begin{equation*}
\begin{split}
 v(\cdot,t)
& =e^{(d_2\Delta -1)t} v_0- \xi \int_0^{t} e^{(d_2\Delta -1)(t-s)}\nabla\cdot(v \nabla u)(\cdot, s)ds\\
 &+ \int_0^{t} e^{(d_2\Delta -1)(t-s)}[v(\mu_2+1-\mu_2v+u-b_2w)](\cdot, s)ds\\
&\leq e^{(d_2\Delta -1)t} v_0- {\xi}\int_0^{t} e^{(d_2\Delta -1)(t-s)}\nabla\cdot(v \nabla u)(\cdot, s)ds\\
&+ \int_0^{t} e^{(d_2\Delta -1)(t-s)}[v(\mu_2+1+u)](\cdot, s)ds,\\
 \end{split}
\end{equation*}
which implies
\begin{equation}\label{LIV-3}
\begin{split}
 \|v(\cdot, t)\|_{L^\infty}
\leq& \|e^{(d_2\Delta -1)t} v_0\|_{L^\infty}+ {\xi}\int_0^{t} \|e^{(d_2\Delta -1)(t-s)}\nabla\cdot(v \nabla u)(\cdot, s)\|_{L^\infty}ds\\
&+ \int_0^{t} \|e^{(d_2\Delta -1)(t-s)}[v(\mu_2+1+u)](\cdot, s)\|_{L^\infty}ds.\\
 \end{split}
\end{equation}
Using \eqref{Lp-3} and  \eqref{LIU-1}, we have
\begin{equation}\label{LIV-4}
\|e^{(d_2\Delta -1)t} v_0\|_{L^\infty}\leq 2\gamma_3 \|v_0\|_{L^\infty},
\end{equation}
and
\begin{equation}\label{LIV-5}
\begin{split}
& \int_0^{t} \|e^{(d_2\Delta -1)(t-s)}[v(\mu_2+1+u)](\cdot, s)\|_{L^\infty}ds\\
&\leq {(\mu_2+1+K)\gamma_3}\int_0^{t}(1+(t-s))^{-\frac{1}{3}}e^{ -(t-s)}\|v(\cdot, s)\|_{L^3}ds\\
&\leq {(\mu_2+1+K)K_8\gamma_3}\int_0^{+\infty}(1+z^{-\frac{1}{3}})e^{-z}dz\\
&=(\mu_2+1+K)K_8\gamma_3[1+\Gamma\left(2/3\right)].
\end{split}
\end{equation}
On the other hand, using the $L^p$-$L^q$ estimate \eqref{Lp-4}  alongside the fact that $C_0^\infty(\Omega)$ is dense in $L^p(\Omega)$ for any $1\leq p<\infty$,  estimates \eqref{LIU-1} and \eqref{LIU*}, we have
\begin{equation}\label{LIV-6}
\begin{split}
 &{\xi}\int_0^{t} \|e^{(d_2\Delta -1)(t-s)}\nabla\cdot(v \nabla u)(\cdot, s)\|_{L^\infty}ds\\
 &\leq {\xi\gamma_4}\int_0^{t}(1+(t-s)^{-\frac{5}{6}})e^{-d_2\lambda_1 t}\|v(\cdot, s)\|_{L^3}\|\nabla u(\cdot, s)\|_{L^\infty}ds\\
 &\leq {\xi\gamma_4K_8K_8^*}\int_0^{\infty}(1+z^{-\frac{5}{6}})e^{-d_2\lambda_1z}dz\\
 &=\frac{\xi\gamma_4K_8K_8^*}{d_2\lambda_1}[1+\Gamma\left(1/6\right)(d_2\lambda_1)^\frac{5}{6}].
 \end{split}
\end{equation}
Substituting \eqref{LIV-4}, \eqref{LIV-5} and  \eqref{LIV-6} into \eqref{LIV-3}, one has \eqref{LIV-1} with
\begin{equation*}
K_9:=2\gamma_3 \|v_0\|_{L^\infty}+(\mu_2+1+K)K_8\gamma_3[1+\Gamma(2/3)]
+
\frac{\xi\gamma_4K_8K_8^*}{d_2\lambda_1}\left(1+\Gamma\left(1/6\right)(d_2\lambda_1)^\frac{5}{6}\right).
\end{equation*}
\end{proof}

\begin{remark}\label{Re3.6}
\em{
From the definitions of $K_4, K_5$ and $K_7$, we see that $\|v(\cdot,t)\|_{L^\infty}\leq K_9 \leq c_0(1+b_1+b_1^2)$ with some constant $c_0>0$ independent of $b_1$. This result will be used later in the stability analysis in section 4.
}
\end{remark}

\subsection{Boundedness of $\|w(\cdot,t)\|_{L^\infty}$} To obtain uniform-in-time boundedness of $\|w(\cdot,t)\|_{L^\infty}$, we first give some higher order derivative estimates of $u, v$ below.
\begin{lemma}\label{L4}
Let the assumptions in Theorem \ref{GB} hold and $(u, v, w)$ be the solution of  \eqref{pd-2}. Then there exists  a constant $K_{10}>0$
independent of $t$ such that
and
for all $p>1$
\begin{equation}\label{L4-2}
\int_t^{t+\tau}\|D^2 u\|_{L^p}^p\leq K_{10}\,\, \mathrm{for\ all} \ t\in (0,\widetilde{T}_{max})
\end{equation}
and
\begin{equation}\label{L4-2*}
{\color{black}\int_\tau^t e^{-p(t-s)}\|\Delta u\|_{L^p}^p\leq K_{10}\,\,  \mathrm{for\ all} \ t\in (\tau,T_{max})},
\end{equation}
where $\tau$ and $\widetilde{T}_{max}$ are defined in \eqref{DTT}.
\end{lemma}
\begin{proof}
We rewrite  the first equation of \eqref{pd-2} as follows
\begin{equation*}
\begin{cases}
u_t-d_1\Delta u+u=F(x,t),&x\in\Omega,t>0,\\
\frac{\partial u}{\partial\nu}=0, &x\in\partial\Omega,t>0,\\
u(x,0)=u_0(x),&x\in\Omega,
\end{cases}
\end{equation*}
where
\begin{equation*}
F(x,t):=\mu_1u(1-u)+u-b_1 uv-\frac{b_3 uw}{u+w}.
\end{equation*}
Then by the boundedness of $|F(x,t)|$(due to the boundedness of $u$ and $v$, see \eqref{Loow} and \eqref{LIV-1}), we can obtain \eqref{L4-2} by  \cite[Lemma 2.3]{JCH-2020-Nonlinearity}. Moreover \eqref{L4-2*} follows from the maximal Sobolev regularity
property, see \cite[Lemma 2.5]{Cao-ZAMP-2016}.
\end{proof}
Next, we shall show that
\begin{lemma}\label{L2-v}
Let the assumptions in Theorem \ref{GB} hold and $(u, v, w)$ be the solution of  \eqref{pd-2}.
Then there exist two positive constants $K_{11}$ and $K_{12}$ independent of $t$ such that
\begin{equation}\label{L2-v1}
\|\nabla v(\cdot, t)\|_{L^2}\leq K_{11}\ \ \mathrm{for\ all} \ t\in (0,T_{max})
\end{equation}
and
\begin{equation}\label{L2-v2}
\int_t^{t+\tau}\|\Delta v\|_{L^2}^2\leq K_{12}\ \ \mathrm{for\ all} \ t\in (0,\widetilde{T}_{max}),
\end{equation}
where $\tau$  and $\widetilde{T}_{max}$ are defined in \eqref{DTT}.
\end{lemma}
\begin{proof}
Multiplying the second equation of \eqref{pd-2} by $-\Delta v$ and integrating the result by parts, we have
\begin{equation}\label{L2-v3}
\begin{split}
&\frac{d}{dt}\int_\Omega|\nabla v|^2+d_2\int_\Omega |\Delta v|^2\\
&=\xi\int_\Omega \nabla \cdot(v\nabla u)\Delta v-\mu_2\int_\Omega v \Delta v+\mu_2\int_\Omega v^2\Delta v-\int_\Omega uv\Delta v+b_2\int_\Omega vw \Delta v.\\
\end{split}
\end{equation}
Noting the facts \eqref{Loow}, \eqref{LIV-1} and \eqref{LIU*}, and using Young's inequality, we can derive that
\begin{equation}\label{L2-v4}
\begin{split}
\xi\int_\Omega \nabla \cdot(v\nabla u)\Delta v
&=\xi\int_\Omega \nabla v\cdot\nabla u \Delta v+\xi\int_\Omega v\Delta u\Delta v\\
&\leq \xi K_8^*\int_\Omega |\nabla v| |\Delta v|+ \xi K_9\int_\Omega|\Delta u||\Delta v|\\
&\leq  \frac{d_2}{2}\int_\Omega |\Delta v|^2+\frac{\xi^2(K_8^*)^2}{d_2}\int_\Omega |\nabla v|^2+\frac{\xi^2 K_9^2}{d_2}\int_\Omega |\Delta  u|^2,
\end{split}
\end{equation}
and
\begin{equation}\label{L2-v5}
\begin{split}
&-\mu_2\int_\Omega v \Delta v+\mu_2\int_\Omega v^2\Delta v-\int_\Omega uv\Delta v+b_2\int_\Omega vw \Delta v\\
&\leq (\mu_2 K_9+\mu_2 K_9^2+KK_9)\int_\Omega|\Delta v|+b_2K_9\int_\Omega w|\Delta v|\\
&\leq \frac{d_2}{4}\int_\Omega |\Delta v|^2+\frac{2b_2^2K_9^2}{d_2} \int_\Omega w^2+C_1,
\end{split}
\end{equation}
with $C_1:=\frac{2(\mu_2 K_9+\mu_2 K_9^2+KK_9)^2|\Omega|}{d_2}$. On the other hand, Young's inequality gives us
\begin{equation}\label{L2-v6}
\begin{split}
\left(\frac{\xi^2(K_8^*)^2}{d_2}+1\right)\int_\Omega |\nabla v|^2
&= - \left(\frac{\xi^2(K_8^*)^2}{d_2}+1\right)\int_\Omega v \Delta v\\
&\leq\left(\frac{\xi^2(K_8^*)^2}{d_2}+1\right)K_9 \int_\Omega |\Delta v|\\
&\leq \frac{d_2}{8}\int_\Omega |\Delta v|^2+C_2
\end{split}
\end{equation}
with $C_2:=\frac{2(\xi^2(K_8^*)^2+d_2)^2K_9^2|\Omega|}{d_2^3}.$
Substituting \eqref{L2-v4}, \eqref{L2-v5} and \eqref{L2-v6} into \eqref{L2-v3}, one has
\begin{equation}\label{L2-v7}
\frac{d}{dt}\int_\Omega|\nabla v|^2+\int_\Omega|\nabla v|^2+\frac{d_2}{8}\int_\Omega |\Delta v|^2
\leq \frac{\xi^2 K_9^2}{d_2}\int_\Omega |\Delta  u|^2+\frac{2b_2^2K_9^2}{d_2}  \int_\Omega w^2+C_1+C_2.
\end{equation}
Using \eqref{L1w-t}  and \eqref{L2u-2}, and applying {\color{black} the Gronwall type inequality (cf. \cite[Lemma 3.4]{SSW14})} to \eqref{L2-v7}, one gets \eqref{L2-v1} directly. Then integrating \eqref{L2-v7} over $(t,t+\tau)$, we obtain
\begin{equation*}
\frac{d_2}{8}\int_{t}^{t+\tau}\int_\Omega |\Delta v|^2
\leq \int_\Omega|\nabla v|^2+ \frac{\xi^2 K_9^2}{d_2}\int_{t}^{t+\tau}\int_\Omega |\Delta  u|^2+\frac{2b_2^2K_9^2}{d_2}\int_{t}^{t+\tau}  \int_\Omega w^2+(C_1+C_2)\tau,
\end{equation*}
which gives \eqref{L2-v2} due to  \eqref{L1w-t}  and \eqref{L2u-2}. Then we complete the proof of this lemma.
\end{proof}
Next, we give the estimate of $\|w(\cdot, t)\|_{L^2}$.
\begin{lemma}\label{L2-w}
Let the assumptions in Theorem \ref{GB} hold and $(u, v, w)$ be the solution of \eqref{pd-2}. Then one can find two positive constants $K_{13}$ and $K_{14}$ independent of $t$ such that
\begin{equation}\label{L2-w1}
\| w(\cdot, t)\|_{L^2}\leq K_{13}\ \ \mathrm{for\ all} \ t\in (0,T_{max})
\end{equation}
and
\begin{equation}\label{L2-w2}
\int_t^{t+\tau} \|\nabla w(\cdot,s)\|_{L^2}^2+\int_t^{t+\tau}\|w(\cdot,s)\|_{L^3}^3\leq K_{14}\ \ \mathrm{for\ all} \ t\in (0,\widetilde{T}_{max}),
\end{equation}
where $\tau$ and $\widetilde{T}_{max}$ are defined in \eqref{DTT}.

\end{lemma}
\begin{proof}
Multiplying the third equation of \eqref{pd-2} by $w$ followed by an  integration by parts, one has
\begin{equation}\label{L2-w3}
\begin{split}
&\frac{1}{2}\frac{d}{dt}\int_\Omega w^2+\int_\Omega |\nabla w|^2+\mu_3 \int_\Omega w^3\\
&=\chi\int_\Omega (wv\nabla u+wu\nabla v)\cdot\nabla w+\mu_3\int_\Omega w^2+\int_\Omega vw^2+c_3\int_\Omega\frac{uw^2}{w+u}.\\
\end{split}
\end{equation}
Since $\|u(\cdot,t)\|_{L^\infty}\leq K$ (see Lemma \ref{UW}), $\|v(\cdot,t)\|_{L^\infty}\leq K_9$ (see Lemma \ref{LIV}) and $\|\nabla u(\cdot,t)\|_{L^\infty}\leq K_8^*$ (see \eqref{LIU*}), one can use Young's inequality to obtain
\begin{equation}\label{L2-w4}
\begin{split}
\chi\int_\Omega (wv\nabla u+wu\nabla v)\cdot\nabla w
&\leq  \chi K_8^*K_9\int_\Omega w|\nabla w|+\chi K\int_\Omega w|\nabla v||\nabla w|\\
&\leq  \frac{1}{2}\int_\Omega |\nabla w|^2+\chi^2 (K_8^*)^2K_9^2\int_\Omega w^2+\chi^2K^2\int_\Omega w^2|\nabla v|^2,
\end{split}
\end{equation}
and
\begin{equation}\label{L2-w5}
\mu_3\int_\Omega w^2+\int_\Omega vw^2+c_3\int_\Omega\frac{uw^2}{w+u}
\leq (\mu_3+K_9+c_3)\int_\Omega w^2.
\end{equation}
Then substituting \eqref{L2-w4} and \eqref{L2-w5} into \eqref{L2-w3}, we obtain
\begin{equation}\label{L2-w6}
\begin{split}
\frac{d}{dt}\int_\Omega w^2+\int_\Omega |\nabla w|^2+2\mu_3 \int_\Omega w^3
&\leq  2\chi^2K^2\int_\Omega w^2|\nabla v|^2+C_1\int_\Omega w^2,\\
\end{split}
\end{equation}
where $C_1:=2[\chi^2 (K_8^*)^2K_9^2+\mu_3+K_9+c_3].$

Using the Gagliardo-Nirenberg inequality in two dimensional spaces, one can find two positive constants $C_2$ and $C_3$ such that
\begin{equation}\label{L2-w7}
\begin{split}
\|w\|_{L^4}^2
&\leq C_2\left(\|\nabla w\|_{L^2}\| w\|_{L^2}
+\| w\|_{L^2}^2\right),
\end{split}
\end{equation}
and (see \cite[Lemma 2.5]{JKW-SIAP-2018})
\begin{equation}\label{L2-w8}
\begin{split}
\|\nabla v\|_{L^4}^2
\leq C_3\left(\|\Delta v\|_{L^2}\|\nabla v\|_{L^2}+\|\nabla v\|_{L^2}^2\right)
\leq C_3K_{11}(\|\Delta v\|_{L^2}+K_{11}),
\end{split}
\end{equation}
where   we have used the fact $\|\nabla v(\cdot, t)\|_{L^2}\leq K_{11}$  in \eqref{L2-v1}. Then we can use \eqref{L2-w7} and \eqref{L2-w8} alongside Young's inequality to obtain
\begin{equation}\label{L2-w9}
\begin{split}
2\chi^2K^2\int_\Omega w^2|\nabla v|^2
\leq &2\chi^2K^2\|w\|_{L^4}^2\|\nabla v\|_{L^4}^2\\
\leq & C_4\left(\|\nabla w\|_{L^2}\| w\|_{L^2}
+\| w\|_{L^2}^2\right)(\|\Delta v\|_{L^2}+K_{11})\\
\leq&  C_4\|\nabla w\|_{L^2}\|w\|_{L^2}\|\Delta v\|_{L^2}+C_4K_{11}\|\nabla w\|_{L^2}\|w\|_{L^2}\\
& +C_4\|w\|_{L^2}^2\|\Delta v\|_{L^2}+C_4K_{11}\|w\|_{L^2}^2\\
\leq&  \frac{1}{2}\|\nabla w\|_{L^2}^2+2C_4^2\|w\|_{L^2}^2\|\Delta v\|_{L^2}^2+\frac{1+4C_{4}^2K_{11}^2}{4}\|w\|_{L^2}^2,
\end{split}
\end{equation}
where $C_4=2\chi^2K^2K_{11}C_2C_3$. Then substituting \eqref{L2-w9} into \eqref{L2-w6} yields
\begin{equation}\label{L2-w10}
\begin{split}
\frac{d}{dt}\|w\|_{L^2}^2+\frac{1}{2}\|\nabla w\|_{L^2}^2+2\mu_3 \|w\|_{L^3}^3
&\leq 2C_4^2\|w\|_{L^2}^2\|\Delta v\|_{L^2}^2+\frac{1+4C_{4}^2K_{11}^2+4C_1}{4}\|w\|_{L^2}^2\\
&\leq C_5 \|w\|_{L^2}^2(\|\Delta v\|_{L^2}^2+1)
\end{split}
\end{equation}
with $C_5:=\frac{1+4C_{4}^2K_{11}^2+4C_1+8C_4^2}{4}$.

Since $\int_t^{t+\tau}\|w\|_{L^2}^2\leq K_3$ (see \eqref{L1w-t}) and $\int_t^{t+\tau}\|\Delta v\|_{L^2}^2\leq K_{12}$ (see \eqref{L2-v2}), by using the similar argument as  in Lemma \ref{L3}, we can obtain from \eqref{L2-w10} that
\begin{equation*}
\|w(\cdot,t)\|_{L^2}^2\leq K_3 e^{C_5K_{12}+C_5},
\end{equation*}
which gives \eqref{L2-w1} with $K_{13}=\left(K_3 e^{C_5K_{12}+C_5}\right)^\frac{1}{2}.$

Integrating \eqref{L2-w10} over $(t,t+\tau)$,  using \eqref{L2-w1} and   $\int_t^{t+\tau}\|\Delta v\|_{L^2}^2\leq K_{12}$ (see \eqref{L2-v2}), one has
\begin{equation*}
\begin{split}
\frac{1}{2}\int_t^{t+\tau}\|\nabla w\|_{L^2}^2+2\mu_3 \int_{t}^{t+\tau}\|w\|_{L^3}^3
&\leq \|w\|_{L^2}^2+C_5\int_t^{t+\tau} \|w\|_{L^2}^2(\|\Delta v\|_{L^2}^2+1)\\
&\leq K_{13}^2+C_5K_{13}^2(K_{12}+1):=K_{14},
\end{split}
\end{equation*}
which gives \eqref{L2-w2}.
\end{proof}


\begin{lemma}\label{L4e}
Let the assumptions in Theorem \ref{GB} hold and $(u, v, w)$ be the solution of system  \eqref{pd-2}. Then there exists  a constant $K_{15}>0$ independent of $t$ such that
\begin{equation}\label{L4e-1}
\|\nabla v(\cdot,t)\|_{L^4}\leq K_{15}\ \ \mathrm{for\ all} \ t\in (0,T_{max}).
\end{equation}
\end{lemma}

\begin{proof}
By the second equation of  \eqref{pd-2}, one has
\begin{equation}\label{L4e-2}
\begin{split}
\frac{1}{4}\frac{d}{dt}\int_\Omega |\nabla v|^4
&=\int_\Omega |\nabla v|^2\nabla v\cdot \nabla v_t\\
&=d_2\int_\Omega |\nabla v|^2\nabla v\cdot \nabla \Delta v-\xi\int_\Omega |\nabla v|^2\nabla v\cdot \nabla (\nabla \cdot(v\nabla u))\\
&\ \ \ \ + \int_\Omega \nabla [\mu_2v(1-v)+uv-b_2 vw]\cdot\nabla v |\nabla v|^2 \\
&:= I_1+I_2+I_3.
\end{split}
\end{equation}
Since $\nabla \Delta v\cdot  \nabla  v=\frac{1}{2}\Delta |\nabla v|^2-|D^2 v|^2$, we can estimate the term $I_1$ as follows:
\begin{equation}\label{I1}
\begin{split}
I_1&=d_2\int_\Omega |\nabla v|^2\nabla v\cdot \nabla \Delta v\\
&=\frac{d_2}{2}\int_\Omega |\nabla v|^2 \Delta |\nabla v|^2-d_2\int_\Omega |\nabla v|^2 |D^2 v|^2\\
&=\frac{d_2}{2}\int_{\partial\Omega} |\nabla v|^2\frac{\partial |\nabla v|^2}{\partial \nu}dS-\frac{d_2}{2}\int_\Omega |\nabla |\nabla v|^2|^2 -d_2\int_\Omega  |\nabla v|^2 |D^2 v|^2.\\
\end{split}
\end{equation}
Using the facts
$\|u(\cdot,t)\|_{L^\infty}\leq K$ (see Lemma \ref{UW}), $\|v(\cdot,t)\|_{L^\infty}\leq K_9$ (see Lemma \ref{LIV}) and $\|\nabla u(\cdot,t)\|_{L^\infty}\leq K_8^*$ (see \eqref{LIU*}), one can estimate terms $I_2$ and $I_3$ as follows:
\begin{equation}\label{I2}
\begin{split}
I_2&=-\xi\int_\Omega |\nabla v|^2\nabla v\cdot \nabla (\nabla \cdot(v\nabla u))\\
&=\xi\int_\Omega \nabla |\nabla v|^2\cdot \nabla v \nabla \cdot (v\nabla u)+\xi\int_\Omega |\nabla v|^2\Delta v\nabla \cdot (v\nabla u)\\
&\leq  \xi\int_\Omega |\nabla v||\nabla |\nabla v|^2|(K_8^*|\nabla v|+K_9|\Delta u|)+\sqrt{2}\xi\int_\Omega |\nabla v|^2|D^2 v|(K_8^*|\nabla v|+K_9|\Delta u|)\\
&\leq \frac{d_2}{4}\int_\Omega |\nabla |\nabla v|^2|^2+\frac{d_2}{2}\int_\Omega |\nabla v|^2|D^2 v|^2+\frac{4\xi^2(K_8^*)^2}{d_2}\int_\Omega|\nabla v|^4+\frac{4\xi^2K_9^2}{d_2}\int_\Omega|\nabla v|^2|\Delta u|^2\\
&\leq \frac{d_2}{4}\int_\Omega |\nabla |\nabla v|^2|^2+\frac{d_2}{2}\int_\Omega |\nabla v|^2|D^2 v|^2+\frac{4\xi^2[(K_8^*)^2+K_9^2]}{d_2}\int_\Omega|\nabla v|^4+\frac{\xi^2 K_9^2}{d_2}\int_\Omega |\Delta u|^4,
\end{split}
\end{equation}
and
\begin{equation}\label{I3}
\begin{split}
I_3&=\int_\Omega \nabla (\mu_2v(1-v)+uv-b_2 vw)\cdot\nabla v |\nabla v|^2\\
&=\mu_2\int_\Omega  |\nabla v|^4-2\mu_2\int_\Omega v|\nabla v|^4 +\int_\Omega v\nabla u\cdot \nabla v |\nabla v|^2+\int_\Omega u|\nabla v|^4\\
&\ \ \ \  -b_2 \int_\Omega w|\nabla v|^4-b_2\int_\Omega v|\nabla v|^2\nabla w\cdot\nabla v\\
&\leq  (\mu_2+K) \int_\Omega |\nabla v|^4+K_8^*K_9\int_\Omega |\nabla v|^3+b_2K_9\int_\Omega |\nabla v|^3|\nabla w|.\\
\end{split}
\end{equation}
Substituting \eqref{I1}-\eqref{I3} into \eqref{L4e-2}, and using Young's inequality, we end up with
\begin{equation}\label{L4e-3}
\begin{split}
&\frac{d}{dt}\int_\Omega |\nabla v|^4+d_2\int_\Omega |\nabla |\nabla v|^2|^2 +2d_2\int_\Omega  |\nabla v|^2 |D^2 v|^2\\
&\leq 2d_2\int_{\partial \Omega} |\nabla v|^2\frac{\partial |\nabla v|^2}{\partial \nu}dS+C_1\int_\Omega |\nabla v|^4+\frac{4\xi^2K_9^2}{d_2}\int_\Omega |\Delta u|^4\\
&\ \ \ \ +4b_2K_9\int_\Omega |\nabla v|^3|\nabla w|+4K_8^*K_9\int_\Omega |\nabla v|^3,
\end{split}
\end{equation}
where $C_1=\frac{16\xi^2[(K_8^*)^2+K_9^2]}{d_2}+4(\mu_2+K).$
With the inequality $\frac{\partial|\nabla v|^2}{\partial \nu}\leq 2\sigma|\nabla v|^2$ on $\partial\Omega$ for some constant $\sigma>0$ (see \cite[Lemma 4.2]{Mizoguchi}), and the trace inequality $\|\varphi\|_{L^2(\partial\Omega)}\leq \varepsilon \|\nabla \varphi\|_{L^2(\Omega)}
+C_\varepsilon\|\varphi\|_{L^2(\Omega)}$ for any $\varepsilon>0$ (see \cite[Remark 52.9]{Souplet}),
 we derive
 \begin{eqnarray}\label{J1}
2d_2\int_{\partial \Omega} |\nabla v|^2\frac{\partial |\nabla v|^2}{\partial \nu}dS
 \leq  4\sigma d_2 \int_{\partial \Omega} |\nabla v|^4dS\leq \frac{d_2}{2}\int_\Omega |\nabla|\nabla v|^2|^2+C_2\int_\Omega |\nabla  v|^4.
\end{eqnarray}
On the other hand, using Young's inequality, we have
\begin{equation}\label{J2}
4K_8^*K_9\int_\Omega |\nabla v|^3\leq C_1 \int_\Omega |\nabla  v|^4+C_3
\end{equation}
with $C_3=\frac{27(K_8^*K_9)^4|\Omega|}{C_1^3}.$
We substitute  \eqref{J1} and \eqref{J2} into \eqref{L4e-3} to obtain
\begin{equation}\label{L4e-4}
\begin{split}
&\frac{d}{dt}\int_\Omega |\nabla v|^4+\frac{d_2}{2}\int_\Omega |\nabla |\nabla v|^2|^2 +2d_2\int_\Omega  |\nabla v|^2 |D^2 v|^2\\
&\leq  (2C_1+C_2)\int_\Omega|\nabla v|^4+\frac{4\xi^2K_9^2}{d_2}\int_\Omega |\Delta u|^4+4b_2K_9\int_\Omega |\nabla v|^3|\nabla w|+C_3.
\end{split}
\end{equation}

Moreover, integrating by parts, noting $\|v(\cdot,t)\|_{L^\infty}\leq K_9$ and  using Young's inequality, one has
\begin{equation*}\label{L4*}
\begin{split}
(2C_1+C_2+2)\int_\Omega |\nabla  v|^4
&=C_4\int_\Omega |\nabla v|^2\nabla v\cdot \nabla v\\
&=-C_4\int_\Omega v \nabla |\nabla v|^2 \cdot \nabla v-C_4\int_\Omega v|\nabla v|^2\Delta v\\
&\leq C_4K_9\int_\Omega |\nabla |\nabla v|^2| |\nabla v|+C_4K_9\sqrt{2}\int_\Omega |\nabla v|^2|D^2 v|\\
&\leq \frac{d_2}{4}\int_\Omega |\nabla |\nabla v|^2|^2+\frac{d_2}{2}\int_\Omega |\nabla v|^2|D^2 v|^2+\frac{2C_4^2K_9^2}{d_2}\int_\Omega |\nabla v|^2\\
&\leq \frac{d_2}{4}\int_\Omega |\nabla |\nabla v|^2|^2+\frac{d_2}{2}\int_\Omega |\nabla v|^2|D^2 v|^4+\int_\Omega |\nabla  v|^4+\frac{C_4^4K_9^4|\Omega|}{d_2^2},\\
\end{split}
\end{equation*}
which gives
\begin{equation}\label{L4e-5}
(2C_1+C_2+1)\int_\Omega |\nabla  v|^4\leq  \frac{d_2}{4}\int_\Omega |\nabla |\nabla v|^2|^2+\frac{d_2}{2}\int_\Omega |\nabla v|^2|D^2 v|^2+\frac{C_4^4K_9^4|\Omega|}{d_2^2}.
\end{equation}
Similarly, one can derive that
\begin{equation*}
\begin{split}
\int_\Omega |\nabla v|^6
&=\int_\Omega |\nabla v|^4\nabla v\cdot\nabla v\\
&=-2\int_\Omega v|\nabla v|^2\nabla|\nabla v|^2\cdot\nabla v-\int_\Omega v|\nabla v|^4\Delta v\\
&\leq 2K_9 \int_\Omega |\nabla v|^3|\nabla |\nabla v|^2|+\sqrt{2}K_9\int_\Omega |\nabla v|^4|D^2 v|\\
&\leq \frac{3}{8}\int_\Omega |\nabla v|^6+4K_9^2\left(\int_\Omega |\nabla |\nabla v|^2|^2+\int_\Omega |\nabla v|^2|D^2 v|^2\right),
\end{split}
\end{equation*}
which entails us that
\begin{equation}\label{L4e-6}
\int_\Omega |\nabla |\nabla v|^2|^2+\int_\Omega |\nabla v|^2|D^2 v|^2\geq \frac{5}{32K_9^2}\int_\Omega |\nabla v|^6.
\end{equation}
Then substituting \eqref{L4e-5} and  \eqref{L4e-6} into \eqref{L4e-4}, we obtain
\begin{equation*}
\begin{split}
&\frac{d}{dt}\int_\Omega |\nabla v|^4+\int_\Omega |\nabla v|^4+\frac{5d_2}{128K_9^2}\int_\Omega |\nabla v|^6\\
&\leq\frac{4\xi^2K_9^2}{d_2}\int_\Omega |\Delta u|^4+4b_2K_9\int_\Omega |\nabla v|^3|\nabla w|+C_3+\frac{C_4^4K_9^4|\Omega|}{d_2^2}\\
&\leq \frac{4\xi^2K_9^2}{d_2}\int_\Omega |\Delta u|^4+\frac{5d_2}{128K_9^2}\int_\Omega |\nabla v|^6+\frac{512b_2^2K_9^4}{5d_2}\int_\Omega|\nabla w|^2+C_3+\frac{C_4^4K_9^4|\Omega|}{d_2^2},
\end{split}
\end{equation*}
and hence
\begin{equation}\label{L4e-7}
\frac{d}{dt}\int_\Omega |\nabla v|^4+\int_\Omega |\nabla v|^4\leq \frac{16\xi^2K_9^2}{d_2}\int_\Omega |D^2 u|^4+\frac{512b_2^2K_9^4}{5d_2}\int_\Omega|\nabla w|^2+C_3+\frac{4C_4^4K_9^4}{4d_2^2},
\end{equation}
where we have use the fact $|\Delta u|\leq \sqrt{2}|D^2 u|$.
Then applying the Gronwall type inequality (cf. \cite[Lemma 3.4]{SSW14}), and using \eqref{L4-2} and \eqref{L2-w2}, from \eqref{L4e-7}, we get \eqref{L4e-1}.

\end{proof}

\begin{lemma}\label{L3-w}
Let the assumptions in Theorem \ref{GB} hold and $(u, v, w)$ be the solution of system  \eqref{pd-2}. Then it holds that
\begin{equation}\label{L3-w1}
\|w(\cdot,t)\|_{L^3}\leq K_{16},
\end{equation}
where $K_{16}>0$ is a constant independent of $t$.
\end{lemma}
\begin{proof}
Multiplying the third equation of \eqref{pd-2} by $w^2$, and integrating the result by parts, we obtain
\begin{equation}\label{L3-w2}
\begin{split}
&\frac{1}{3}\frac{d}{dt}\int_\Omega w^3+2\int_\Omega w|\nabla w|^2+\mu_3 \int_\Omega w^4\\
&=2\chi\int_\Omega w^2(v\nabla u+u\nabla v)\cdot\nabla w+\mu_3\int_\Omega w^3+\int_\Omega vw^3+c_3\int_\Omega\frac{uw^3}{w+u}.
\end{split}
\end{equation}
By the facts $\|u(\cdot,t)\|_{L^\infty}\leq K$ (see Lemma \ref{UW}), $\|v(\cdot,t)\|_{L^\infty}\leq K_9$ (see Lemma \ref{LIV}), $\|\nabla u(\cdot,t)\|_{L^\infty}\leq K_8^*$ (see \eqref{LIU*}) and $\|\nabla v\|_{L^4}\leq K_{15}$ (see Lemma \ref{L4e}), one can derive that
\begin{equation}\label{L3-w3}
\begin{split}
&2\chi\int_\Omega w^2(v\nabla u+u\nabla v)\cdot\nabla w\\
&\leq
 2\chi K_8^*K_9\int_\Omega w^2 |\nabla w|+2\chi K\int_\Omega w^2|\nabla v||\nabla w|\\
 &\leq  \int_\Omega w|\nabla w|^2+2 \chi^2 (K_8^*)^2K^2_9\int_\Omega w^3+2\chi^2K^2\int_\Omega w^3|\nabla v|^2,\\
\end{split}
\end{equation}
and
\begin{equation}\label{L3-w4}
\begin{split}
\mu_3\int_\Omega w^3+\int_\Omega vw^3+c_3\int_\Omega\frac{uw^3}{w+u}\leq (\mu_3+K_9+c_3) \int_\Omega w^3.
\end{split}
\end{equation}
Substituting \eqref{L3-w3} and \eqref{L3-w4} into
\eqref{L3-w2} gives
\begin{equation}\label{L3-w5}
\begin{split}
\frac{d}{dt}\int_\Omega w^3+3\int_\Omega w|\nabla w|^2+3\mu_3 \int_\Omega w^4
\leq 6\chi^2K^2\int_\Omega w^3|\nabla v|^2+C_1\int_\Omega w^3,
\end{split}
\end{equation}
with $C_1:=6 \chi^2 (K_8^*)^2K^2_9+3(\mu_3+K_9+c_3).$ Then using the Gagliardo-Nirenberg inequality and Young's inequality, and utilizing the facts $\|w(\cdot,t)\|_{L^2}\leq K_{13}$ (see \eqref{L2-w1}) and  $\|\nabla v\|_{L^4}\leq K_{15}$ (see  Lemma \ref{L4e}), one has
\begin{equation}\label{L3-w6}
\begin{split}
6\chi^2K^2\int_\Omega w^3|\nabla v|^2
&\leq  6\chi^2K^2\left(\int_\Omega w^6\right)^\frac{1}{2}\left(\int_\Omega |\nabla v|^4\right)^\frac{1}{2}\\
&\leq 6\chi^2K^2K_{15}^2\|w^\frac{3}{2}\|_{L^4}^2\\
&\leq 6\chi^2K^2K_{15}^2C_2(\|\nabla w^\frac{3}{2}\|_{L^2}^\frac{4}{3}
 \|w^\frac{3}{2}\|_{L^\frac{4}{3}}^\frac{2}{3}+\|w^\frac{3}{2}\|_{L^\frac{4}{3}}^2)\\
 &=6\chi^2K^2K_{15}^2C_2K_{13} \|\nabla w^\frac{3}{2}\|_{L^2}^\frac{4}{3}+6\chi^2K^2K_{15}^2C_2K^3_{13}\\
 &\leq \int_\Omega w|\nabla w|^2+C_3.
\end{split}
\end{equation}
Then substituting \eqref{L3-w6} into \eqref{L3-w5}, and adding $\int_\Omega w^3$ on both sides of the resulting inequality alongside the Young's inequality: $(C_1+1)\int_\Omega w^3\leq \mu_3 \int_\Omega w^4+C_4$,  we obtain
\begin{equation*}
\frac{d}{dt}\int_\Omega w^3+\int_\Omega w^3\leq C_3+C_4,
\end{equation*}
 which gives \eqref{L3-w1} directly upon an application of Gr\"{o}nwall's inequality.
\end{proof}

Next, we shall show the boundedness $\|\nabla v(\cdot, t)\|_{L^\infty}$. Precisely, we have

\begin{lemma}\label{GL-v}
Let the assumptions in Theorem \ref{GB} hold and $(u, v, w)$ be the solution of  \eqref{pd-2}. Then there exists  a constant $K_{17}>0$ independent of $t$ such that for all $t\in (0,T_{max})$
\begin{equation}\label{GL-v1}
\|\nabla v(\cdot, t)\|_{L^\infty}\leq K_{17}.
\end{equation}
\end{lemma}

\begin{proof}
{\color{black}
From Lemma \ref{LS}, we know that \eqref{GL-v1} holds for all $t\in(0,\tau]$, where $\tau$ is defined by \eqref{DTT}. Hence to prove Lemma \ref{GL-v}, we only need to show  \eqref{GL-v1} holds for all $t\in (\tau,T_{max})$.}

To this end, we rewrite the second equation of \eqref{pd-2} as follows:
\begin{equation}\label{GL-v2}
v_t-d_2\Delta v+v=-\xi\nabla\cdot(v\nabla u)+\mu_2v(1-v)+uv-b_2 vw+v.
\end{equation}
With the variation-of-constants formula for  \eqref{GL-v2}, we obtain
\begin{equation*}\label{GL-v3}
\begin{split}
\nabla v(\cdot,t)
=&\nabla e^{(d_2\Delta-1)(t-\tau) }v(\cdot,\tau)-\xi \int_{\tau}^t\nabla e^{(d_2\Delta-1)(t-s)}[\nabla\cdot(v\nabla u)]ds\\
&+\int_\tau^t\nabla e^{(d_2\Delta-1)(t-s)}[\mu_2v(1-v)+uv-b_2 vw+v]ds,
\end{split}
\end{equation*}
which gives
\begin{equation}\label{ell}
\begin{split}
\|\nabla v(\cdot,t)\|_{L^\infty}
\leq &\|\nabla e^{(d_2\Delta-1)(t-\tau)}v(\cdot,\tau)\|_{L^\infty}+\xi \int_\tau^t\|\nabla e^{(d_2\Delta-1)(t-s)}[\nabla\cdot(v\nabla u)]\|_{L^\infty}ds\\
&+\int_\tau^t\|\nabla e^{(d_2\Delta-1)(t-s)}[\mu_2v(1-v)+uv-b_2 vw+v]\|_{L^\infty}ds\\
:=&\ell_1+\ell_2+\ell_3.
\end{split}
\end{equation}
Using the semigroup smoothing estimates in Lemma \ref{heatsemigroup},
we first estimate the term $\ell_1$ as
\begin{equation}\label{ell-1}
\ell_1=\|\nabla e^{(d_2\Delta-1) (t-\tau)}v(\cdot,\tau)\|_{L^\infty}\leq C_1 ,
\end{equation}
and estimate the term $\ell_2$ as
\begin{equation}\label{ell-2}
\begin{split}
\ell_2
&= \xi \int_\tau^t\|\nabla e^{(d_2\Delta-1)(t-s)}(\nabla v\cdot\nabla u+v\Delta u)\|_{L^\infty}ds\\
&\leq  \xi \int_\tau^t\|\nabla e^{(d_2\Delta-1)(t-s)}(\nabla v\cdot\nabla u)\|_{L^\infty}ds+\xi \int_\tau^t\|\nabla e^{(d_2\Delta-1)(t-s)}(v\Delta u)\|_{L^\infty}ds\\
&:= \ell_{21}+\ell_{22}.
\end{split}
\end{equation}
Then
from Lemma \ref{heatsemigroup}, Lemma \ref{LIV}, Lemma \ref{L4e} and \eqref{LIU*}, it follows that
\begin{equation}\label{ell-21}
\begin{split}
\ell_{21}
&= \xi \int_\tau^t\|\nabla e^{(d_2\Delta-1)(t-s)}(\nabla v\cdot\nabla u)\|_{L^\infty}ds\\
&\leq  C_2 \int_\tau^t(1+(t-s)^{-\frac{1}{2}-\frac{1}{4}})e^{-(d_2\lambda_1+1)(t-s)}\|\nabla v\cdot \nabla u\|_{L^4}ds\\
&\leq  C_2 \int_\tau^t(1+(t-s)^{-\frac{3}{4}})e^{-(d_2\lambda_1+1)(t-s)}\|\nabla v\|_{L^4}\| \nabla u\|_{L^\infty}ds\\
&\leq  C_3 \int_\tau^t(1+(t-s)^{-\frac{3}{4}})e^{-(d_2\lambda_1+1)(t-s)}ds\\
&\leq  C_4,
\end{split}
\end{equation}
and
\begin{equation}\label{ell-22}
\begin{split}
\ell_{22}
&=\xi \int_\tau^t\|\nabla e^{(d_2\Delta-1)(t-s)}(v\Delta u)\|_{L^\infty}ds\\
&\leq  C_5 \int_\tau^t(1+(t-s)^{-\frac{1}{2}-\frac{1}{p}})e^{-(d_2\lambda_1+1)(t-s)}\|v\Delta u\|_{L^p}ds \\
&\leq C_6 \int_\tau^t(1+(t-s)^{-\frac{1}{2}-\frac{1}{p}})
e^{-(d_2\lambda_1+1)(t-s)}\|\Delta u\|_{L^p}ds\\
&\leq  C_7 \int_\tau^t \left(1+(t-s)^{-\frac{1}{2}-\frac{1}{p}}\right)^\frac{p}{p-1}
e^{-\frac{d_2\lambda_1p}{p-1}(t-s)}ds+C_7\int_\tau^t e^{-p(t-s)}\|\Delta u\|_{L^p}^pds.\\
\end{split}
\end{equation}
Choosing $p>4$, we can check that $\frac{p+2}{2(p-1)}<1$, and hence
\begin{equation*}
\int_\tau^t \left(1+(t-s)^{-\frac{1}{2}-\frac{1}{p}}\right)^\frac{p}{p-1}
e^{-\frac{d_2\lambda_1p}{p-1}(t-s)}ds\leq C_8 \int_\tau^t \left(1+(t-s)^{-\frac{p+2}{2(p-1)}}\right)
e^{-\frac{d_2\lambda_1p}{p-1}(t-s)}ds\leq C_9,
\end{equation*}
which, alongside \eqref{L4-2*},  gives
$
\ell_{22}\leq C_{10}
$
for some constant $C_{10}>0$ from \eqref{ell-22}. This entails that
\begin{equation}\label{ell-2*}
\ell_2\leq C_4+C_{10}.
\end{equation}
Finally using the boundedness of $\|v(\cdot, t)\|_{L^\infty}, \|u(\cdot, t)\|_{L^\infty}$ and $\|w(\cdot, t)\|_{L^3}$, we get the estimate for $\ell_3$ as follows:
\begin{equation}\label{ell-3}
\begin{split}
\ell_3&=\int_\tau^t\|\nabla e^{(d_2\Delta-1)(t-s)}[\mu_2v(1-v)+uv-b_2 vw+v]\|_{L^\infty}ds\\
&\leq C_{11}\int_\tau^t(1+(t-s)^{-\frac{1}{2}-\frac{1}{3}})e^{-(d_2\lambda_1+1)(t-s)}\|\mu_2v(1-v)+uv-b_2 vw+v\|_{L^3}ds\\
&\leq C_{12}.
\end{split}
\end{equation}
Then substituting \eqref{ell-1}, \eqref{ell-2*} and \eqref{ell-3} into \eqref{ell}, we obtain \eqref{GL-v1}. The proof of Lemma \ref{GL-v} is completed.
\end{proof}
\begin{lemma}\label{GL-w}
Let the assumptions in Theorem \ref{GB} hold and $(u, v, w)$ be the solution of system  \eqref{pd-2}. Then one can find  a constant $K_{18}>0$ independent of $t$ such that
\begin{equation}\label{GL-w1}
\|w(\cdot, t)\|_{L^\infty}\leq K_{18} \ \ \mathrm{for\ all} \ t\in (0,T_{max}).
\end{equation}
\end{lemma}
\begin{proof}
By  the variation of constants formula, $w$ can be represented as
 \begin{equation*}
 \begin{split}
w(\cdot, t)=&e^{(\Delta-1) t } w_0- \chi \int_0^t e^{(\Delta-1)(t-s)}\nabla\cdot(wv \nabla u+wu \nabla v)\\
&+ \int_0^t e^{(\Delta-1)(t-s)}w\left(\mu_3+1-\mu_3w-v+c_3\frac{u}{u+w}\right)\\
\leq& e^{(\Delta-1) t } w_0- \chi \int_0^t e^{(\Delta-1)(t-s)}\nabla\cdot(wv \nabla u+wu \nabla v)\\
& +\int_0^t e^{(\Delta-1)(t-s)}w\left(\mu_3+1+c_3\frac{u}{u+w}\right)
\end{split}
\end{equation*}
and hence
\begin{equation}\label{GL-w2}
 \begin{split}
\|w(\cdot, t)\|_{L^\infty}\leq &\|e^{(\Delta-1) t } w_0\|_{L^\infty}+\chi \int_0^t \|e^{(\Delta-1)(t-s)}\nabla\cdot(wv \nabla u+wu \nabla v)\|_{L^\infty}\\
&+ \int_0^t \Big\| e^{(\Delta-1)(t-s)}w\left(\mu_3+1+c_3\frac{u}{u+w}\right)\Big\|_{L^\infty}:=J_1+J_2+J_3.
\end{split}
\end{equation}
Using the well-known semigroup  smoothing estimates (see Lemma \ref{heatsemigroup}), we have
\begin{equation}\label{GL-w3}
J_1=\|e^{(\Delta-1) t } w_0\|_{L^\infty}\leq C_1\|w_0\|_{L^\infty}
\end{equation}
for some constant $C>0$. Noting the facts
$\|u(\cdot,t)\|_{L^\infty}\leq K$ (see Lemma \ref{UW}), $\|v(\cdot,t)\|_{L^\infty}\leq K_9$ (see Lemma \ref{LIV}), $\|\nabla u(\cdot,t)\|_{L^\infty}\leq K_8^*$ (see \eqref{LIU*}),
$\|\nabla v(\cdot,t)\|_{L^\infty}\leq K_{17}$ (see Lemma \ref{GL-v}) and $\|w(\cdot,t)\|_{L^3}\leq K_{16} $ (see Lemma \ref{L3-w}), one can use \eqref{Lp-4} with the fact that $C_0^\infty(\Omega)$ is dense in $L^p(\Omega)$ for any $1\leq p<\infty$ to obtain
\begin{equation}\label{GL-w4}
\begin{split}
J_2&\leq \chi \gamma_4 \int_0^t (1+(t-s)^{-\frac{5}{6}})e^{-(\lambda_1+1)(t-s)}\|w(v \nabla u+u \nabla v)\|_{L^3}ds\\
&\leq \chi \gamma_4(K_9K_8^*+KK_{17})\int_0^t (1+(t-s)^{-\frac{5}{6}})e^{-(\lambda_1+1)(t-s)}\|w\|_{L^3}ds\\
&\leq \chi \gamma_4(K_9K_8^*+KK_{17})K_{16}\int_0^t (1+(t-s)^{-\frac{5}{6}})e^{-(\lambda_1+1)(t-s)}ds\\
&\leq \frac{\chi \gamma_4(K_9K_8^*+KK_{17})K_{16}}{\lambda_1}(1+\Gamma\left(1/6\right)\lambda_1^\frac{5}{6}).
\end{split}
\end{equation}
Moreover, we can use \eqref{Lp-3} to derive that
\begin{equation}\label{GL-w5}
\begin{split}
J_3&\leq  \gamma_3(\mu_3+1+c_3)\int_0^t (1+(t-s)^{-\frac{1}{3}})e^{-(t-s)}\|w\|_{L^3}ds\\
&\leq \gamma_3(\mu_3+1+c_3)K_{16} \int_0^t (1+(t-s)^{-\frac{1}{3}})e^{-(t-s)}ds\\
&\leq \gamma_3(\mu_3+1+c_3)K_{16}(1+\Gamma(2/3)).
\end{split}
\end{equation}
Then substituting \eqref{GL-w3}, \eqref{GL-w4} and \eqref{GL-w5} into  \eqref{GL-w2}, one obtains \eqref{GL-w1}.
\end{proof}	

\begin{proof}
[Proof of Theorem \ref{GB}]
The combination of Lemma \ref{UW}, \eqref{LIU*}, Lemma \ref{LIV}, Lemma \ref{GL-v} and Lemma \ref{GL-w}, yields a constant $C_1>0$ independent of $t$ such that
\begin{equation*}
\|u(\cdot,t)\|_{W^{1,\infty}}+\|v(\cdot,t)\|_{W^{1,\infty}}+\|w(\cdot,t)\|_{L^{\infty}}\leq C_1,
\end{equation*}
which together with  the extension criterion in Lemma \ref{LS} proves Theorem \ref{GB}.
\end{proof}

\section{Global stabilization of solutions}
In this section, we are devoted to  studying the global stability of coexistence steady states as asserted in Theorem \ref{GS} and Theorem \ref{GS1} by the Lyapunov functional method along with Barb\v{a}lat's Lemma as stated below.
\begin{lemma}[Barb\u{a}lat's Lemma \cite{Barbalat1959}] \label{barbalet}
If $h: [1, \infty) \to \mathbb{R}$ is a uniformly continuous function such that
$\lim\limits_{t \to \infty} \int_1^t h(s)ds$ exists, then $\lim\limits_{t \to \infty}h(t)=0$.
\end{lemma}
Moreover, we need higher regularity of solutions as follows.
\begin{lemma}\label{le3-2}
Let $(u, v, w)$ be the unique global bounded classical solution of \eqref{pd-2} given by Theorem \ref{GB}.
Then for any given $0<\alpha <1,$ there exists a constant $C>0$ such that
\begin{eqnarray*}\label{us1}
\|u(\cdot,t)\|_{C^{2+\alpha, 1+\frac{\alpha}{2}}(\bar \Omega \times [1, \infty))}+\|v(\cdot,t)\|_{C^{2+\alpha, 1+\frac{\alpha}{2}}(\bar \Omega \times [1, \infty))}+\|w(\cdot,t)\|_{C^{2+\alpha, 1+\frac{\alpha}{2}}(\bar \Omega \times [1, \infty))}\leq C.
\end{eqnarray*}
\end{lemma}
\begin{proof}
By the boundedness of $(u, v, w)$ (see Theorem \ref{GB})  and the interior $L^p$ estimate (\cite{Lieberman-1996}), we find a constant $C_1>0$ such that
\begin{eqnarray}\label{eq3-1}
\|u\|_{W_p^{2, 1}(\Omega \times [i+\frac{1}{4}, i+3])}+\|v\|_{W_p^{2, 1}(\Omega \times [i+\frac{1}{4}, i+3])}+\|w\|_{W_p^{2, 1}(\Omega \times [i+\frac{1}{4}, i+3])}\leq C_1,\, \forall i \geq 0.
\end{eqnarray}
 Then the Sobolev embedding theorem with $p$ suitably large gives
  \begin{eqnarray}\label{eq3-2}
\|u\|_{C^{1+\alpha, \frac{1+\alpha}{2}}(\bar \Omega \times [\frac{1}{4}, \infty))}+\|v\|_{C^{1+\alpha, \frac{1+\alpha}{2}}(\bar \Omega \times [\frac{1}{4}, \infty))}+\|w\|_{C^{1+\alpha, \frac{1+\alpha}{2}}(\bar \Omega \times [\frac{1}{4}, \infty))}\leq C_2.
\end{eqnarray}
 Using \eqref{eq3-2} and applying the  Schauder estimate \cite{La-book-1968}  to the first equation of \eqref{pd-2}, we obtain
\begin{eqnarray}\label{eq3-3}
\|u\|_{C^{2+\alpha, 1+\frac{\alpha}{2}}(\bar \Omega \times [i+\frac{1}{3}, i+3])}\leq C_3,\, \forall i \geq 0,
\end{eqnarray}
and hence
\begin{eqnarray}\label{eq3-4}
\|u\|_{C^{2+\alpha, 1+\frac{\alpha}{2}}(\bar \Omega \times [\frac{1}{3}, +\infty))}\leq C_4.
\end{eqnarray}
The second equation of \eqref{pd-2} can be rewritten as
\begin{eqnarray}\label{eq3-5}
v_t-d_2\Delta v+\xi\nabla u\cdot \nabla v =G(x, t), \,\, x\in \Omega,\,\, t>0,
\end{eqnarray}
where
$$
G(x, t)=-\xi \Delta u \cdot v+v(\mu_2-\mu_2v+u-b_2 w).
$$
By \eqref{eq3-3} and \eqref{eq3-4}, we have
$$\|G\|_{C^{\alpha, \frac{\alpha}{2}}(\bar \Omega \times [i+\frac{1}{3}, i+3])}+ \|\xi\nabla u\|_{C^{\alpha, \frac{\alpha}{2}}(\bar \Omega \times [i+\frac{1}{3}, i+3])} \leq C_5,\, \forall i \geq 0.$$
Then applying the standard parabolic Schauder estimate to \eqref{eq3-5}, one can find a positive constant $C_6>0$ such that $\|v\|_{C^{2+\alpha, 1+\frac{\alpha}{2}}(\bar\Omega \times [i+1, i+3])} \leq C_6$ for all $i \geq 0$,
 which gives
\begin{eqnarray}\label{eq3-6}
\|v\|_{C^{2+\alpha, 1+\frac{\alpha}{2}}(\bar\Omega \times [1, +\infty))} \leq C_7.
\end{eqnarray}
Finally, similar arguments applied to the third equation of \eqref{pd-2} give us a constant  $C_8>0$ such that
$
\|w\|_{C^{2+\alpha, 1+\frac{\alpha}{2}}(\bar\Omega \times [1, +\infty))} \leq C_8,
$
which completes the proof of Lemma \ref{le3-2}.

\end{proof}

\subsection{Global stability for $b_3=c_3=0$} In this subsection, we first show the global stabilization of coexistence steady state in the case of $b_3=c_3=0$. In this case, the system \eqref{pd-2-1} becomes
\begin{equation}\label{pd-2-2}
\begin{cases}
u_t=d_1 \Delta u+u(1-u-b_1 v),&x\in\Omega, t>0,\\
v_t=d_2\Delta v-\xi\nabla\cdot(v\nabla u)+v(1-v+u-b_2 w),&x\in\Omega, t>0,\\
w_t=\Delta w-\chi\nabla \cdot(wv\nabla u+wu\nabla v)+ w(1-w+v),&x\in\Omega, t>0,\\
\frac{\partial u}{\partial \nu}=\frac{\partial v}{\partial \nu}=\frac{\partial w}{\partial \nu}=0, &x\in \partial\Omega,\,\, t>0,\\
u(x,0)=u_0(x), \,\, v(x, 0)=v_0(x), \,\, w(x, 0)=w_0(x),&x\in \Omega.
\end{cases}
\end{equation}
One can easily check  that the system \eqref{pd-2-2} has a unique constant coexistence steady state $(u^*,v^*, w^*)$ defined by \eqref{uvw*} provided
\begin{equation}\label{GS1-C1}
0<b_2<2 \ \ \mathrm{and}\ \ \ 1+b_1b_2+b_2>b_1.
\end{equation}
Our purpose is to study the global stability of $(u^*,v^*,w^*)$. To this end, we introduce  the following energy functional:
\begin{eqnarray}\label{eq3-8}
\mathcal{E}_1(t)=\mathcal{I}_u(t)+\mathcal{I}_v(t)+\mathcal{I}_w(t),
\end{eqnarray}
where
$$
\mathcal{I}_s(t)=\int_\Omega \left(s-s^*-s^*\ln \frac{s}{s^*}\right), \ s\in \{u,v,w\}.
$$
The energy functional like \eqref{eq3-8} is the Lyapunov functional to the corresponding ODE (cf. \cite{Braun, Hsu-survey}), and it can be extended to PDE models with diffusion (cf. \cite{Du-Shi, Yi-Wei-Shi}) or prey-taxis (cf. \cite{Ahn-Yoon-JDE-2020, JW-JDE-2017}).

Then we have the following results:
\begin{lemma}\label{le3-3}
Let  the condition \eqref{GS1-C1} hold. If
\begin{equation}\label{GS1-C2}
(b_1-1)^2+(b_2-1)^2<4,
\end{equation}
then there exist $\xi_1>0$ and $\chi_1>0$ such that whenever $\xi \in (0,\xi_1)$ and $\chi \in (0, \chi_1)$ there holds  for some $T_0>0$ that
\begin{equation}\label{GS1-C4}
\|u(\cdot,t)-u^*\|_{L^\infty}+\|v(\cdot,t)-v^*\|_{L^\infty}+\|w(\cdot,t)-w^*\|_{L^\infty}\leq K_{19} e^{-\lambda t}, \mathrm{\ for \ all} \  t>T_0,
\end{equation}
where $K_{19}>0$ and $\lambda>0$ are constants independent of $t$.

\end{lemma}

\begin{proof}
The coexistence steady state $(u^*,v^*,w^*)$ satisfies equations
\begin{equation}\label{CE}
\begin{cases}
1-u^*-b_1 v^*=0,\\
1-v^*+u^*-b_2 w^*=0,\\
1-w^*+v^*=0.\\
\end{cases}
\end{equation}
Then one can check that \eqref{CE} has a unique positive steady state $(u^*,v^*,w^*)$ defined by \eqref{uvw*}  under the condition \eqref{GS1-C1}.

{\bf{Step 1.}} In this step, we shall show the global stability of $(u^*,v^*,w^*)$ by means of the energy functional $\mathcal{E}_1(t)$.
Using the first equation of \eqref{pd-2-2} and the
fact that $1-u^*-b_1 v^*=0,$ we find
\begin{equation}\label{eq3-12}
\begin{split}
\frac{d}{dt}\mathcal{I}_u(t)&= \int_\Omega \left(1-\frac{u^*}{u}\right)u_t\\
&=-d_1 u^*\int_\Omega \frac{|\nabla u|^2}{u^2}+\int_\Omega(u-u^*)(1-u-b_1v)\\
&=-d_1 u^*\int_\Omega \frac{|\nabla u|^2}{u^2}-\int_\Omega(u-u^*)^2-b_1\int_\Omega(u-u^*)(v-v^*).
\end{split}
\end{equation}
Similarly, we can use the second equation of \eqref{pd-2-2} and the fact $1-v^*+u^*-b_2 w^*=0$ to obtain
\begin{equation}\label{eq3-13}
\begin{split}
\frac{d}{dt}\mathcal{I}_v(t) &= \int_\Omega \left(1-\frac{v^*}{v}\right)v_t\\
&=-d_2 v^*\int_\Omega \frac{|\nabla v|^2}{v^2}+\xi v^*\int_\Omega\frac{\nabla u\cdot \nabla v}{v}+\int_\Omega (v-v^*)(1-v+u-b_2w)\\
&=-d_2 v^*\int_\Omega \frac{|\nabla v|^2}{v^2}+\xi v^*\int_\Omega\frac{\nabla u\cdot \nabla v}{v}-\int_\Omega(v-v^*)^2+\int_\Omega(u-u^*)(v-v^*)\\
&\ \ \ \ -b_2\int_\Omega(v-v^*)(w-w^*).
\end{split}
\end{equation}
Furthermore with the fact $1-w^*+v^*=0$ and using the third equation of \eqref{pd-2-2}, we have
\begin{equation}\label{eq3-14}
\begin{split}
\frac{d}{dt}\mathcal{I}_w(t)&= \int_\Omega \left(1-\frac{w^*}{w}\right)w_t\\
&=-w^*\int_\Omega \frac{|\nabla w|^2}{w^2}+\chi w^*\int_\Omega\frac{v \nabla u\cdot \nabla w+u\nabla v\cdot \nabla w}{w}+\int_\Omega (w-w^*)(1-w+v)\\
&=-w^*\int_\Omega \frac{|\nabla w|^2}{w^2}+\chi w^*\int_\Omega\frac{v \nabla u\cdot \nabla w+u\nabla v\cdot \nabla w}{w}\\
&\ \ \ \ -\int_\Omega (w-w^*)^2+\int_\Omega (v-v^*)(w-w^*).
\end{split}
\end{equation}
Then substituting \eqref{eq3-12},  \eqref{eq3-13} and \eqref{eq3-14} into \eqref{eq3-8}, we end up with
\begin{equation}\label{eq3-15}
\frac{d}{dt}\mathcal{E}_1(t)=-\int_\Omega X A_1 X^T-\int_\Omega Y B_1 Y^T,
\end{equation}
 where $X=(u-u^{*},v-v^{*},w-w^{*})$, $Y=\left(\frac{\nabla u}{u},\frac{\nabla v}{v},\frac{\nabla w}{w}\right)$
and $A_1, B_1$ are symmetric matrices denoted by
\begin{equation*}
A_1:=
\begin{pmatrix}
1&\frac{b_{1}-1}{2}&0\\
\frac{b_1-1}{2}&1&\frac{b_2-1}{2}\\
0&\frac{b_2-1}{2}&1
\end{pmatrix}, \ \
B_1:=
\begin{pmatrix}
d_1 u^{*}&-\frac{\xi v^* u}{2}&-\frac{\chi w^*uv}{2}\\
-\frac{\xi v^* u}{2}&d_2v^*&-\frac{\chi w^*uv}{2}\\
-\frac{\chi w^*uv}{2}&-\frac{\chi w^*uv}{2}&w^{*}
\end{pmatrix}.
\end{equation*}

Next, we show that the matrices $A_1$ and $B_1$ are positive definite and positive semi-definite respectively. Notice $\eqref{GS1-C2}$  implies  $0<b_1<3$. Then
\begin{equation*}
  \begin{vmatrix}
  1&\frac{b_{1}-1}{2}\\
\frac{b_1-1}{2}&1
  \end{vmatrix}=\frac{(3-b_1)(1+b_1)}{4}> 0,
\end{equation*}
and	
	\begin{equation*}
|A_1|=\frac{4-(b_1-1)^2-(b_2-1)^2}{4}>0.
\end{equation*}
Therefore, $A_1$ is positive definite and there exists a constant $\alpha>0$ such that
\begin{eqnarray}\label{eq3-16}
X A_1 X^T\geq \alpha |X|^2.
\end{eqnarray}	
On the other hand, after some calculations, one can derive that
\begin{equation*}
  \begin{vmatrix}
 d_1 u^{*}&-\frac{\xi v^* u}{2}\\
-\frac{\xi v^* u}{2}&d_2v^*
  \end{vmatrix}=\frac{v^*(4d_1 d_2u^*-\xi^2v^*u^2)}{4}\geq \frac{v^*(4d_1 d_2u^*-\xi^2v^*\|u\|_{L^\infty}^2)}{4}
\end{equation*}
and		
\begin{equation*}
\begin{split}
|B_1|=&-\frac{w^*}{4}\left[\xi\chi^2v^*w^*u^3v^2+\chi^2w^*(d_1 u^*+d_2v^*)u^2v^2+\xi^2 (v^*)^2u^2-4d_1 d_2u^* v^*\right]\\
\geq&-\frac{w^*}{4}\left[\xi\chi^2v^*w^*\|u\|_{L^\infty}^3\|v\|_{L^\infty}^2+\chi^2w^*(d_1 u^*+d_2v^*)\|u\|_{L^\infty}^2 \|v\|_{L^\infty}^2\right]\\
&-\frac{w^*}{4}\left[\xi^2 (v^*)^2\|u\|_{L^\infty}^2-4d_1 d_2u^* v^*\right].\\
\end{split}
\end{equation*}
Noticing that $\|u\|_{L^\infty}$ and $\|v\|_{L^\infty}$ are independent of parameters $\xi$ and $\chi$ (see Theorem \ref{GB}), we can find appropriate numbers $\xi_1>0$ and $\chi_1>0$, fox example,
$$\xi_1= \sqrt{\frac{2d_1 d_2u^*}{ v^*\|u\|_{L^\infty}^2}} \ \text{and}  \ \chi_1=\sqrt{{2d_1 d_2u^* v^*}/{\|u\|_{L^\infty}^2 \|v\|_{L^\infty}^2(\xi v^*w^*\|u\|_{L^\infty}+w^*(d_1 u^*+d_2v^*))}}$$
such that if $0<\xi<\xi_1, 0<\chi<\chi_1$, then
$$4d_1 d_2u^* v^*>\xi\chi^2v^*w^*\|u\|_{L^\infty}^3\|v\|_{L^\infty}^2+\chi^2w^*(d_2v^*+d_1 u^*)\|u\|_{L^\infty}^2 \|v\|_{L^\infty}^2+\xi^2 (v^*)^2\|u\|_{L^\infty}^2,$$
which guarantees that  $B$ is a positive semi-definite matrix, and hence	
	\begin{eqnarray}\label{eq3-17}
Y B_1 Y^T\geq 0.
\end{eqnarray}
Substituting \eqref{eq3-16}	and \eqref{eq3-17} into \eqref{eq3-15}, we obtain
\begin{equation}\label{EF1t}
\frac{d}{dt}\mathcal{E}_1(t)+\alpha \mathcal{F}_1(t)\leq 0
\end{equation}
with
\begin{equation*}\label{F1t}
\mathcal{F}_1(t):= \int_\Omega (u-u^*)^2+(v-v^*)^2+(w-w^*)^2.
\end{equation*}
Moreover, we can show that $\mathcal{E}_1(t)\geq 0$ for all $t>0.$ In fact, letting $\varphi(z):=z-u^* \ln z$ for $z>0,$ one can check that $\varphi'(z)=1-\frac{u^*}{z}$ and $\varphi''(z)=\frac{u^*}{z^2}$. By the Taylor's expansion, there exists a quantity $\eta=\theta u+(1-\theta)u^*$  with $\theta\in(0,1)$ such that
\begin{equation}\label{TL-1}
  u-u^{*}-u^{*}\ln\frac{u}{u^{*}} =\varphi(u)-\varphi(u^*)
  =\frac{\varphi''(\eta )}{2}(u-u^{*})^2
  =\frac{u^{*}}{2\eta^2}(u-u^{*})^2\geq 0,
\end{equation}
which implies $\mathcal{I}_{u}(t)\geq 0$. Similarly, we have that $\mathcal{I}_{v}(t)\geq0$ and $\mathcal{I}_{w}(t)\geq0$. Then it follows that $\mathcal{E}_{1}(t)=\mathcal{I}_{u}(t)+\mathcal{I}_{v}(t)+\mathcal{I}_{w}(t)\geq 0$.

%
Then integrating \eqref{EF1t} with respect of $t$ over $(1,\infty)$ along with $\mathcal{E}_1(t)\geq 0$, gives
\begin{equation}\label{F1t-I}
\int_1^\infty \mathcal{F}_{1}(t) \leq \frac{1}{\alpha}\mathcal{E}_{1} (1) <\infty.
\end{equation}
Then using the regularity of $u, v, w$ obtained in Lemma \ref{le3-2}, one can derive that $\mathcal{F}_{1}(t)$ is uniformly continuous in $[1, \infty)$. Then using \eqref{F1t-I} and applying Lemma \ref{barbalet}, we obtain
\begin{equation*}
 \mathcal{F}_{1}(t)=\int_{\Omega}(u-u^{*})^2+\int_{\Omega}(v-v^{*})^2+\int_{\Omega}(w-w^{*})^2 \to 0 \mbox{ as } t \to \infty,
\end{equation*}
which gives
\begin{equation}\label{eq3-19}
\lim\limits_{t\to\infty}(\|u-u^*\|_{L^2}+\|v-v^*\|_{L^2}+\|w-w^*\|_{L^2})=0.
\end{equation}
By Lemma \ref{le3-2}, we find
\begin{equation}\label{eq3-20}
\|u-u^*\|_{W^{1,\infty}}+\|v-v^*\|_{W^{1,\infty}}+\|w-w^*\|_{W^{1,\infty}}\leq C_1, \mbox{ for all } t>1.
\end{equation}
Then applying the Gagliardo-Nirenberg inequality and using \eqref{eq3-20}, one has
\begin{equation}\label{eq3-20*}
\|u-u^*\|_{L^\infty}\leq C_2 \|u-u^*\|_{W^{1,\infty}}^\frac{1}{2}\|u-u^*\|_{L^2}^\frac{1}{2}\leq C_2 C_1^\frac{1}{2}\|u-u^*\|_{L^2}^\frac{1}{2},
\end{equation}
which together with \eqref{eq3-19} implies
\begin{equation}\label{eq3-21}
\lim\limits_{t\to \infty}\|u-u^*\|_{L^\infty}=0.
\end{equation}
The same argument gives us that
\begin{equation}\label{eq3-22}
\lim\limits_{t\to \infty}(\|v-v^*\|_{L^\infty}+\|w-w^*\|_{L^\infty})=0.
\end{equation}
{\bf Step 2.}
In this step, we shall show that the convergence rate is exponential. In fact, \eqref{eq3-21} implies there exists $t_1>1$ such that
\begin{equation*}
\frac{u^*}{2}\leq \eta=[\theta u+(1-\theta)u^*]\leq \frac{3u^*}{2},\,\, \mbox{ for all }\, t\geq t_1.
\end{equation*}
Then by \eqref{TL-1}, one has
\begin{equation}\label{eq3-23}
\frac{2}{9u^*}(u-u^*)^2\leq u-u^{*}-u^{*}\ln\frac{u}{u^{*}}\leq \frac{2}{u^*}(u-u^*)^2.
\end{equation}
Similarly, using \eqref{eq3-22}, there exist two constants $t_2>1$  and $t_3>1$ such that
\begin{equation}\label{eq3-24}
\frac{2}{9v^*}(u-u^*)^2\leq v-v^{*}-v^{*}\ln\frac{v}{v^{*}}\leq \frac{2}{v^*}(v-v^*)^2, \mbox{ for all } t>t_2,
\end{equation}
and
\begin{equation}\label{eq3-25}
\frac{2}{9w^*}(w-w^*)^2\leq w-w^{*}-w^{*}\ln\frac{w}{w^{*}}\leq \frac{2}{w^*}(w-w^*)^2, \mbox{ for all } t>t_3.
\end{equation}
By the definition of $\mathcal{E}_1(t)$ and $\mathcal{F}_1(t)$ along with \eqref{eq3-23}, \eqref{eq3-24} and \eqref{eq3-25}, we can find two positive constants $\alpha_1$ and $\alpha_2$ such that
\begin{equation}\label{eq3-26}
\alpha_1 \mathcal{F}_{1}(t)\leq\mathcal{E}_1(t)\leq \alpha_2 \mathcal{F}_{1}(t)
\end{equation}
for some $t\geq \bar t=\max\{t_1,t_2, t_3\}.$   Then the combination of \eqref{EF1t} and \eqref{eq3-26} gives
\begin{equation*}
\frac{d}{dt}\mathcal{E}_1(t)+\frac{\alpha}{\alpha_2} \mathcal{E}_1(t)\leq 0, \ \mathrm{for\ all}\ \ t\geq \bar t,
\end{equation*}
which implies for all $t\geq \bar {t}$
\begin{equation}\label{eq3-27}
\mathcal{E}_1(t)\leq \mathcal{E}_1(\bar t)e^{-\frac{\alpha}{\alpha_2}(t-\bar t)}\leq C_3 e^{-\frac{\alpha}{\alpha_2}t}.
\end{equation}
Then it follows from  \eqref{eq3-26}-\eqref{eq3-27} that
\begin{equation*}
\mathcal{F}_{1}(t)\leq \frac{1}{\alpha_1}\mathcal{E}_1(t)\leq \frac{C_3}{\alpha_1}e^{-\frac{\alpha}{\alpha_2}t}, \mbox{ for all } t\geq \bar t,
\end{equation*}
which, alongside the definition of $\mathcal{F}_1(t)$, gives
\begin{equation}\label{eq3-28}
  \|u-u^{*}\|_{L^2}^2+\|v-v^{*}\|_{L^2}^2+\|w-w^{*}\|_{L^2}^2
  \leq \frac{C_3}{\alpha_1}e^{-\frac{\alpha}{\alpha_2}t}, \mbox{ for all } t\geq \bar t.
\end{equation}
Then combining \eqref{eq3-20*} and \eqref{eq3-28}, one can find there exists two positive constant $C_4$ and $\alpha_3$ such that
\begin{equation} \label{eq3-29}
\|u(\cdot,t)-u^*\|_{L^\infty}\leq C_4 e^{-\alpha_3 t},\,\, \mbox{ for all } t\geq \bar t.
\end{equation}
 Similarly, it holds that
\begin{equation}\label{eq3-30}
\|v(\cdot,t)-v^*\|_{L^\infty}+\|w(\cdot,t)-w^*\|_{L^\infty}\leq C_5 e^{-\alpha_4 t},\,\, \ \mathrm{for \ all}\ \ t\geq \bar t.
\end{equation}
Combining  \eqref{eq3-29} and \eqref{eq3-30}, gives \eqref{GS1-C4} and hence completes  the proof.
\end{proof}

\subsection{Global stability for $b_3>0,\,\,c_3>0$}

In this subsection, we investigate the large time behavior of solutions for the system \eqref{pd-2-1} with  $b_3>0$ and $c_3>0.$ For simplicity, we assume $c_3=1.$ We underline  from Remark \ref{Re3.6}  that $\|v\|_{L^\infty}$ is bounded by a constant independent of $b_1$ as $b_1$ is small. Hence if $b_1>0$  and $b_3>0$ are suitably small, $1-b_3-b_1\|v\|_{L^\infty}>0$ is warranted.
To derive the stability in the case of $b_3, c_3>0,$  we first derive a lower bound estimate for $u.$
\begin{lemma}\label{le4-0}
Let $(u, v, w)$ be the unique global bounded classical solution of \eqref{pd-2-1}. Let  $b_1>0$ and $b_3>0$ be sufficiently small such that
$
 1-b_3-b_1 \|v\|_{L^\infty}>0.
$
Then there exists a constant $t_0\in (0, \infty)$ such that
\begin{eqnarray}\label{eq4.0-1}
u(x, t)\geq \min\{\bar u, 1-b_3-b_1 \|v\|_{L^\infty}\}\,\,\, \mbox{ for all } (x, t) \in \bar \Omega \times (t_0, \infty),
\end{eqnarray}
and
\begin{eqnarray}\label{eq4.0-2}
\mathop{\lim \inf}\limits_{t\to \infty}u(x, t)\geq 1-b_3-b_1 \|v\|_{L^\infty} \mbox{ for all } x \in \bar \Omega,
\end{eqnarray}
where $\bar u=\mathop{\min}\limits_{x\in \bar\Omega}u(x, t_0).$
\end{lemma}

\begin{proof}
By the  strong maximum principle applied to the first equation in \eqref{pd-2-1}, there exists  a constant  $t_0>0$ such that
$\mathop{\min}\limits_{x\in \bar\Omega}u(x, t_0)=\bar u>0$. For convenience, we denote $\mathcal{K}=1-b_1 \|v\|_{L^\infty}-b_3.$
 Then we consider the following problem
\begin{equation}\label{eq4.1}
\begin{cases}
u_t-d_1 \Delta u=u\left(1-u-b_1 v- \frac{b_3 w}{u+w}\right)\geq u(\mathcal{K}-u),&x\in\Omega, t>t_0,\\
\frac{\partial u}{\partial \nu}=0,&x\in\partial\Omega, t>t_0,\\
u|_{t=t_0}=u(x, t_0), &x \in \Omega.
\end{cases}
\end{equation}
Let $\tilde u(t)$ be the solution of the following ODE problem
\begin{equation*}\label{eq4.2}
\begin{cases}
\frac{d\tilde u(t)}{dt}=\tilde u(\mathcal{K}-\tilde u),& t>t_0,\\
\tilde u(t_0)=\bar u>0,
\end{cases}
\end{equation*}
which has the explicit solution $\tilde u(t)=\frac{\mathcal{K}}{1+(\frac{\mathcal{K}}{\bar u}-1)e^{-\mathcal{K}(t-t_0)}}$ such that
$$
\tilde u(t)\geq \min\{\bar u, \mathcal{K}\},\,\, t>t_0.
$$
Recall that $u(x, t_0)\geq \bar u.$ Then $\tilde{u}$ is a lower solution of the following PDE problem
\begin{equation}\label{eq4.3}
\begin{cases}
U_t^0-d_1 \Delta U^0= U^0(\mathcal{K}-U^0),&x\in\Omega, t>t_0,\\
\frac{\partial U^0}{\partial \nu}(x, t)=0,&x\in\partial\Omega, t>t_0,\\
U^0(x, t_0)=u(x, t_0),& x\in \Omega
\end{cases}
\end{equation}
and consequently
\begin{eqnarray}\label{eq4.4}
\tilde u(t)\leq U^0(x, t)\,\, \mbox{ for all } (x, t) \in \bar \Omega \times (t_0, \infty).
\end{eqnarray}
Applying the comparison principle to  \eqref{eq4.1} and \eqref{eq4.3}, and using \eqref{eq4.4}, one has
\begin{eqnarray*}\label{eq4.5}
\min\{\bar u, \mathcal{K}\} \leq \tilde u(t) \leq U^0(x, t)\leq u(x, t)\,\, \mbox{ for all } (x, t) \in \bar \Omega \times (t_0, \infty),
\end{eqnarray*}
which indicates \eqref{eq4.0-1} and \eqref{eq4.0-2}.
\end{proof}

Next, we shall show that the coexistence steady state $(u_*,v_*,w_*)$ is globally asymptotically stable under some conditions. Similarly, we introduce the following energy functional:
\begin{eqnarray}\label{eq4-1}
\mathcal{E}_2(t)=\mathcal{J}_u(t)+\mathcal{J}_v(t)+\mathcal{J}_w(t),
\end{eqnarray}
where
$$
\mathcal{J}_s(t)=\int_\Omega \left(s-s_*-s_*\ln \frac{s}{s_*}\right), \ s\in \{u,v,w\},
$$
and $(u_*, v_*, w_*)=(u_1^*, v_1^*, w_1^*)$ is defined in \eqref{u1star}. Then we have the following results.
\begin{lemma}\label{GS2}
Let $(u,v,w)$ be the solution of \eqref{pd-2-1}, and $(u_*, v_*, w_*)=(u_1^*, v_1^*, w_1^*)$ be the coexistence steady state  defined in \eqref{u1star}. If \eqref{condition3-s1} holds, then there exist $\xi_2>0,\chi_2>0$ such that the solution $(u,v,w)$ will exponentially converge to $(u_*, v_*, w_*)$ in $L^\infty$-norm as $t \to \infty$ whenever $0<\xi<\xi_2$ and $0<\chi<\chi_2$.
\end{lemma}

\begin{proof}
Using the first equation of \eqref{pd-2-1} and the
fact that $1-u_*-b_1 v_*-b_3\frac{w_*}{u_*+w_*}=0$ in \eqref{ce-1},  we derive
\begin{equation}\label{eq4-4}
\begin{split}
\frac{d}{dt} \mathcal{J}_u(t)
&= \int_\Omega \left(1-\frac{u_*}{u}\right)u_t\\
&=-d_1 u_*\int_\Omega \frac{|\nabla u|^2}{u^2}+\int_\Omega(u-u_*)\left(1-u-b_1v-\frac{b_3 w}{u+w}\right)\\
&=-d_1 u_*\int_\Omega \frac{|\nabla u|^2}{u^2}-\int_\Omega(u-u_*)^2-b_1\int_\Omega(u-u_*)(v-v_*)\\
&\ \ \ \ -b_3\int_\Omega \frac{u_*}{(u_*+w_*)(u+w)}(u-u_*)(w-w_*)+b_3\int_\Omega \frac{w_*}{(u_*+w_*)(u+w)}(u-u_*)^2.
\end{split}
\end{equation}
With $1-v_*+u_*-b_2w_*=0$  in \eqref{ce-1}, we can use the second equation of \eqref{pd-2-1} to derive that
\begin{equation}\label{eq4-5}
\begin{split}
\frac{d}{dt} \mathcal{J}_v(t)
& = \int_\Omega \left(1-\frac{v_*}{v}\right)v_t\\
&=-d_2 v_*\int_\Omega \frac{|\nabla v|^2}{v^2}+\xi v_*\int_\Omega\frac{\nabla u\cdot \nabla v}{v}+\int_\Omega (v-v_*)(1-v+u-b_2w)\\
&=-d_2 v_*\int_\Omega \frac{|\nabla v|^2}{v^2}+\xi v_*\int_\Omega\frac{\nabla u\cdot \nabla v}{v}-\int_\Omega(v-v_*)^2+\int_\Omega(u-u_*)(v-v_*)\\
&\ \ \ \ -b_2\int_\Omega(v-v_*)(w-w_*).
\end{split}
\end{equation}
Similarly, using the third equation of \eqref{pd-2-1} with $c_3=1$ and the fact $1-w_*+v_*+\frac{u_*}{u_*+ w_*}=0$  in \eqref{ce-1}, we obtain
\begin{equation}\label{eq4-6}
\begin{split}
&\frac{d}{dt} \mathcal{J}_w(t)
 =  \int_\Omega \left(1-\frac{w_*}{w}\right)w_t\\
&=-w_*\int_\Omega \frac{|\nabla w|^2}{w^2}+\chi w_*\int_\Omega\frac{v \nabla u\cdot \nabla w+u\nabla v\cdot \nabla w}{w}+\int_\Omega (w-w_*)\left(1-w+v+\frac{u}{u+w}\right)\\
&=-w_*\int_\Omega \frac{|\nabla w|^2}{w^2}+\chi w_*\int_\Omega\frac{v \nabla u\cdot \nabla w+u\nabla v\cdot \nabla w}{w}
-\int_\Omega (w-w_*)^2+\int_\Omega (v-v_*)(w-w_*)\\
&\ \ \ \ +\int_\Omega\frac{w_*}{(u_*+w_*)(u+w)}(u-u_*)(w-w_*)-\int_\Omega\frac{u_*}{(u_*+w_*)(u+w)}(w-w_*)^2.
\end{split}
\end{equation}
Using the definition of $\mathcal{E}_2(t)$ in \eqref{eq4-1}, and the identities   \eqref{eq4-4}-\eqref{eq4-6}, we have
\begin{eqnarray}\label{eq4-7}
\frac{d}{dt} \mathcal{E}_2(t)=-\int_\Omega X_1 A_2 X_1^T-\int_\Omega Y_1 B_2 Y_1^T,
\end{eqnarray}
where $X_1=(u-u_*, v-v_*, w-w_*)$, $Y_1=\left(\frac{\nabla u}{u}, \frac{\nabla v}{v}, \frac{\nabla w}{w}\right)$ and $A_2, B_2$ are symmetric
matrices denoted by
\begin{equation*}
A_2:=
\begin{pmatrix}
1-\frac{b_3w_*}{(u_*+w_*)(u+w)}&\frac{b_{1}-1}{2}&\frac{b_3u_*-w_*}{2(u_*+w_*)(u+w)}\\
\frac{b_1-1}{2}&1&\frac{b_2-1}{2}\\
\frac{b_3u_*-w_*}{2(u_*+w_*)(u+w)}&\frac{b_2-1}{2}&1+\frac{u_*}{(u_*+w_*)(u+w)}
\end{pmatrix}, \ \
B_2:=
\begin{pmatrix}
d_1 u_{*}&-\frac{\xi v_* u}{2}&-\frac{\chi w_*uv}{2}\\
-\frac{\xi v_* u}{2}&d_2v_*&-\frac{\chi w_*uv}{2}\\
-\frac{\chi w_*uv}{2}&-\frac{\chi w_*uv}{2}&w_*
\end{pmatrix}.
\end{equation*}
Recalling that $\|u\|_{L^\infty}$ and $\|v\|_{L^\infty}$ are independent of parameters $\xi$ and $\chi$ (see Theorem \ref{GB}), one can find appropriate numbers $\xi_2>0$ and $\chi_2>0$ such that if $\xi\in(0,\xi_2)$ and $\chi \in (0,\chi_2)$ then
$$4d_1 d_2u_* v_*> \xi\chi^2v_*w_*\|u\|_{L^\infty}^3\|v\|_{L^\infty}^2+\chi^2w_*(d_2v_*+d_1 u_*)\|u\|_{L^\infty}^2 \|v\|_{L^\infty}^2+\xi^2 (v_*)^2\|u\|_{L^\infty}^2$$
which gives rises to
\begin{equation*}
  \begin{vmatrix}
 d_1 u_{*}&-\frac{\xi v_* u}{2}\\
-\frac{\xi v_* u}{2}&d_2v_*
  \end{vmatrix}=\frac{v_*(4d_1 d_2u_*-\xi^2v_*u^2)}{4}\geq \frac{v_*(4d_1 d_2u_*-\xi^2v_*\|u\|_{L^\infty}^2)}{4}> 0,
\end{equation*}
and		
\begin{equation*}
\begin{split}
|B_2|&=-\frac{w_*}{4}\left[\xi\chi^2v_*w_*u^3v^2+\chi^2w_*(d_1 u_*+d_2v_*)u^2v^2+\xi^2 v_*^2u^2-4d_1 d_2u_* v_*\right]\\
&\geq -\frac{w_*}{4}\left[\xi\chi^2v^*w_*\|u\|_{L^\infty}^3\|v\|_{L^\infty}^2+\chi^2w_*(d_1 u_*+d_2v_*)\|u\|_{L^\infty}^2 \|v\|_{L^\infty}^2+\xi^2 (v_*)^2\|u\|_{L^\infty}^2-4d_1 d_2u_* v_*\right]\\
&> 0.
\end{split}
\end{equation*}
These imply that the matrix $B_2$ is positive definite and hence
\begin{eqnarray}\label{eq4-8}
Y_1 B_2 Y_1^T> 0.
\end{eqnarray}
Next, we  claim
\begin{itemize}
\item If $\frac{1}{10}\leq b_2<\sqrt 2$ and $b_1 = b_3=0$, then the matrix $A_2$ is positive definite.
\end{itemize}
In fact, if $b_1=b_3=0$ and $0<b_2<\sqrt{2}$, one can check that the system \eqref{pd-2-1} with $c_3=1$ has a unique positive coexistence steady state $(\bar u_*,\bar v_*, \bar w_*)$ satisfying (see also \eqref{solu} in the Appendix)
\begin{equation}\label{ce-00}
\begin{cases}
\bar u_*=1,\\
\bar v_*=\frac{b_2^2+2b_2+4-b_2 \sqrt{(b_2+2)(b_2+10)}}{2(b_2+1)},\\
\bar w_*=\frac{2-b_2+\sqrt{(b_2+2)(b_2+10)}}{2(b_2+1)}.
\end{cases}
\end{equation}
Moreover, when $b_1=b_3=0$, the corresponding matrix $A_2$ becomes
\begin{equation*}
\tilde{A}_2:=
\begin{pmatrix}
1&-\frac{1}{2}&-\frac{\bar w_*}{2(\bar u_*+\bar w_*)(u+w)}\\
-\frac{1}{2}&1&\frac{b_2-1}{2}\\
-\frac{\bar w_*}{2(\bar u_*+\bar w_*)(u+w)}&\frac{b_2-1}{2}&1+\frac{\bar u_*}{(\bar u_*+\bar w_*)(u+w)}
\end{pmatrix},
\end{equation*}
which is positive definite if $|\tilde{A}_2|>0.$ After some calculations, one can check that
\begin{equation}\label{ce-00-1}
4|\tilde{A}_2|=3-(b_2-1)^2-\frac{(\bar w_*)^2}{(\bar u_*+\bar w_*)^2(u+w)^2}+\frac{3\bar u_*+(b_2-1)\bar w_*}{(\bar u_*+\bar w_*)(u+w)}.
\end{equation}
From Lemma \ref{le4-0}, we know that if $b_1=b_3=0$, it holds that
\begin{equation*}
\liminf\limits_{t\to\infty} u(x,t)\geq 1,
\end{equation*}
which implies there exists  $T_1>0$ such that $u(x,t)\geq \frac{\sqrt{2}}{2}$ for all $t\geq T_1$ and hence
\begin{equation}\label{ce-00-2}
u(x,t)+w(x,t)\geq \frac{\sqrt{2}}{2} \ \ \mathrm{for\  all}\ \ x\in\Omega \ \mathrm{and}\ t\geq T_1.
\end{equation}
 With \eqref{ce-00-2} in hand and using the facts  $\bar u_*=1$ and $0<b_2<\sqrt{2}$, we can directly calculate that
 \begin{equation}\label{ce-00-3}
 \begin{split}
 3-(b_2-1)^2-\frac{(\bar w_*)^2}{(\bar u_*+\bar w_*)^2(u+w)^2}
 \geq 3-(b_2-1)^2-2\frac{(\bar w_*)^2}{(1+\bar w_*)^2}
 \geq 1-(b_2-1)^2>0.
 \end{split}
 \end{equation}
On the other hand,  if $b_2\geq \frac{1}{10}$, one can derive that
 \begin{equation*}
 \begin{split}
3\bar u_*+(b_2-1)\bar w_*
=3+(b_2-1)\bar w_*
=\frac{9b_2+4-b_2^2+(b_2-1)\sqrt{(b_2+2)(b_2+10)}}{2(b_2+1)}>0,
\end{split}
 \end{equation*}
 and hence
 \begin{equation}\label{ce-00-4}
 \frac{3\bar u_*+(b_2-1)\bar w_*}{(\bar u_*+\bar w_*)(u+w)}>0.
 \end{equation}
Combining \eqref{ce-00-1}, \eqref{ce-00-3} and \eqref{ce-00-4}, one finds that  the matrix $\tilde{A}_2$ is positive definite if $\frac{1}{10}\leq b_2<\sqrt{2}$. By the continuity of $(u_*, v_*, w_*)$ with respect to $b_1$  and $b_3$  (see Remark \ref{Re5.1}), if $b_1$  and $b_3$ are small enough and $\frac{1}{10}\leq b_2<\sqrt{2}$, $A_2$ is positive definite for all $t\geq T_1$. Then there exists a constant $\beta >0$ such that
 \begin{equation}\label{eq4-9}
X_1 A_2 X_1^T\geq \beta X_1^2\,\, \mbox{ for all } t\geq T_1.
 \end{equation}
Combining \eqref{eq4-7}, \eqref{eq4-8} and \eqref{eq4-9}, one has
\begin{equation*}\label{eq4-10}
\frac{d}{dt}\mathcal{E}_2(t)+ \beta\mathcal{F}_2(t)\leq 0\,\, \ \mathrm{for\ all}\ \ t\geq T_1,
\end{equation*}
where
\begin{equation*}\label{F2t}
\mathcal{F}_2(t):=\int_\Omega (u-u_*)^2+(v-v_*)^2+(w-w_*)^2.
\end{equation*}
Using the similar argument as in the proof of Lemma \ref{le4-0},  we can show
\begin{equation*}
\|u(\cdot,t)-u_*\|_{L^\infty}+\|v(\cdot,t)-v_*\|_{L^\infty}+\|u(\cdot,t)-u_*\|_{L^\infty}\leq C e^{-\sigma t} \mathrm{\ for \ all} \  t>T_2
\end{equation*}
hold for some positive constants $C, \sigma$ and $T_2.$
\end{proof}
\begin{proof}[Proof of Theorem \ref{GS} and Theorem \ref{GS1}] Theorem \ref{GS} is a consequence of Lemma \ref{le3-3} and  Theorem \ref{GS1} results from Lemma \ref{GS2}.
\end{proof}

%
%


\section{Appendix}
The homogeneous coexistence steady state $(u_*,v_*,w_*)$ in \eqref{pd-2-1} satisfies the following equations
\begin{equation*}
    \begin{cases}
    1-u_*-b_1 v_*-b_3 \frac{w_*}{u_*+w_*}=0,\\
    1-v_*+u_*-b_2 w_*=0,\\
    1-w_*+v_*+\frac{u_*}{u_*+w_*}=0,\\
    \end{cases}
    \end{equation*}
which has two explicit solutions
$(u_1^*, v_1^*, w_1^*)$ and $(u_2^*, v_2^*, w_2^*)$ as follows:

\begin{equation}\label{u1star}
\begin{cases}
u_1^* =& \frac{2 b_1^2 ( b_2^2-1) +
 b_1 ( b_2+2) (b_2 - 3 b_3 + 2 b_2 b_3) + (1 + b_2) (2 + ( b_2-2) b_3 -
    b_3^2)}{2 (1 + b_1 + b_2) (1 + b_1 (1 + b_2) + (2 + b_2) b_3)}
\\[2mm]
&+\frac{(b_1b_2+b_2b_3+b_3) \sqrt{
        20 + 12 b_2 + b_2^2 + 4 b_1^2 ( b_2^2-1) +
         4 b_1 (3 + b_2) (b_2 - b_3) - 20 b_3 - 14 b_2 b_3 + b_3^2}}{2 (1 + b_1 + b_2) (1 + b_1 (1 + b_2) + (2 + b_2) b_3)},\\
 \\
v_1^*=&\frac{4 + 2 b_2 + b_2^2 + 2 b_3 - 4 b_2 b_3 - 4 b_2^2 b_3 - b_3^2 -
   2 b_1 (1 + b_2) ( b_2 + b_3-2)}
     {2 (1 + b_1 +b_2) (1 + b_1 (1 + b_2) + (2 + b_2) b_3)}\\[2mm]
   &+  \frac{
   (b_3-b_2) \sqrt{20 + 12 b_2 + b_2^2 + 4 b_1^2 ( b_2^2-1) +
     4 b_1 (3 + b_2) (b_2 - b_3) - 20 b_3 - 14 b_2 b_3 + b_3^2}}
     {2 (1 + b_1 +b_2) (1 + b_1 (1 + b_2) + (2 + b_2) b_3)},\\
     \\
w_1^*=&\frac{2 + 4 b_1 + 2 b_1^2 - b_2 + 5 b_1 b_2 + 2 b_1^2 b_2 + 9 b_3 + 5 b_1 b_3 +
   7 b_2 b_3 + 2 b_1 b_2 b_3 -
   b_3^2 }
      {2 (1 + b_1 + b_2) (1 +
     b_1 (1 + b_2) + (2 + b_2) b_3)}  \\[2mm]
   &+\frac{(1 + b_1 + b_3) \sqrt{
    20 + b_2^2 + 4 b_1^2 ( b_2^2-1) + 4 b_1 ( b_2+3) (b_2 - b_3) - 20 b_3 +
      b_3^2 - 2 b_2 ( 7 b_3-6)}}
      {2 (1 + b_1 + b_2) (1 +
     b_1 (1 + b_2) + (2 + b_2) b_3)},  \\
\end{cases}
\end{equation}

and

\begin{equation}\label{u2star}
\begin{cases}
u_2^* =& \frac{2 b_1^2 ( b_2^2-1) +
 b_1 ( b_2+2) (b_2 - 3 b_3 + 2 b_2 b_3) + (1 + b_2) (2 + ( b_2-2) b_3 -
    b_3^2)}{2 (1 + b_1 + b_2) (1 + b_1 (1 + b_2) + (2 + b_2) b_3)}
\\[2mm]
&-\frac{(b_1b_2+b_2b_3+b_3) \sqrt{
        20 + 12 b_2 + b_2^2 + 4 b_1^2 ( b_2^2-1) +
         4 b_1 (3 + b_2) (b_2 - b_3) - 20 b_3 - 14 b_2 b_3 + b_3^2}}{2 (1 + b_1 + b_2) (1 + b_1 (1 + b_2) + (2 + b_2) b_3)},\\
 \\
v_2^*=&\frac{4 + 2 b_2 + b_2^2 + 2 b_3 - 4 b_2 b_3 - 4 b_2^2 b_3 - b_3^2 -
   2 b_1 (1 + b_2) ( b_2 + b_3-2)}
     {2 (1 + b_1 +b_2) (1 + b_1 (1 + b_2) + (2 + b_2) b_3)}\\[2mm]
   &-  \frac{
   (b_3-b_2) \sqrt{20 + 12 b_2 + b_2^2 + 4 b_1^2 ( b_2^2-1) +
     4 b_1 (3 + b_2) (b_2 - b_3) - 20 b_3 - 14 b_2 b_3 + b_3^2}}
     {2 (1 + b_1 +b_2) (1 + b_1 (1 + b_2) + (2 + b_2) b_3)},\\
     \\
w_2^*=&\frac{2 + 4 b_1 + 2 b_1^2 - b_2 + 5 b_1 b_2 + 2 b_1^2 b_2 + 9 b_3 + 5 b_1 b_3 +
   7 b_2 b_3 + 2 b_1 b_2 b_3 -
   b_3^2 }
      {2 (1 + b_1 + b_2) (1 +
     b_1 (1 + b_2) + (2 + b_2) b_3)}  \\[2mm]
   &-\frac{(1 + b_1 + b_3) \sqrt{
    20 + b_2^2 + 4 b_1^2 ( b_2^2-1) + 4 b_1 ( b_2+3) (b_2 - b_3) - 20 b_3 +
      b_3^2 - 2 b_2 ( 7 b_3-6)}}
      {2 (1 + b_1 + b_2) (1 +
     b_1 (1 + b_2) + (2 + b_2) b_3)}.  \\
\end{cases}
\end{equation}

\begin{remark}\label{Re5.1}
It can be seen from \eqref{u1star} and \eqref{u2star} that $(u_1^*, v_1^*, w_1^*)$  and $(u_2^*, v_2^*, w_2^*)$ are continuous with respect
to $b_1, b_2$ and $ b_3$  for any $b_1, b_2, b_3 \geq 0.$  If $b_1=b_3=0,$ we have
\begin{equation}\label{solu}
\begin{cases}
u_1^* =1,\\
v_1^*=\frac{b_2^2+2b_2+4-b_2 \sqrt{(b_2+2)(b_2+10)}}{2(b_2+1)},\\
w_1^*=\frac{2-b_2+\sqrt{(b_2+2)(b_2+10)}}{2(b_2+1)},\\
\end{cases}
\end{equation}
and
\begin{equation*}
\begin{cases}
u_2^* =1,\\
v_2^*=\frac{b_2^2+2b_2+4+b_2 \sqrt{(b_2+2)(b_2+10)}}{2(b_2+1)},\\
w_2^*=\frac{2-b_2-\sqrt{(b_2+2)(b_2+10)}}{2(b_2+1)}.\\
\end{cases}
\end{equation*}
Clearly, if $0<b_2<\sqrt 2$, one has
\begin{equation}\label{eqff}
u_1^*, v_1^*, w_1^*>0 \mbox{ and } u_2^*>0, v_2^*>0, w_2^*<0.
\end{equation}
Therefore, by the continuity,  we still have \eqref{eqff} if $b_1$  and $b_3$
are small enough and $0<b_2<\sqrt 2$. In this case,  we only  have one positive coexistence steady state $(u_1^*, v_1^*, w_1^*).$
\end{remark}

\noindent \textbf{Acknowledgments.}
 We are deeply grateful to  the referees for carefully reading our paper and giving us various corrections and very insightful suggestions/comments, which greatly improve the precision of our results and exposition of our manuscript.
The research of H.Y. Jin was supported by
the NSF of China (No. 11871226) and the Fundamental Research Funds for the Central Universities. The research of Z.A. Wang  was partially supported by the Hong Kong RGC GRF grant No. 15306121 and the 2020 Hong Kong Scholars Program. The research of L. Wu  was supported by the 2020 Hong Kong Scholars Program  (Project ID P0031250 and Primary Work Programme YZ3S).


\begin{thebibliography}{99}
\bibitem{Abrahams}
M.V. Abrahams and L.D. Townsend.  Bioluminescence in dinoflagellates: A test of the burgular alarm hypothesis. {\it Ecology}, 74(1):258--260, 1993.

\bibitem{Ahn-Yoon-JDE-2020} I. Ahn and C. Yoon. Global well-posedness and stability analysis of prey-predator model with indirect prey-taxis. {\it J. Differential Equations}, 268:4222-4255, 2020.


 \bibitem{ABN-2008} B.E. Ainseba, M. Bendahmane and A. Noussair. A reaction-diffusion system modeling predator-prey with prey-taxis. {\it Nonlinear Anal. Real World Appl., }9(5):2086--2105, 2008.

\bibitem{Amann2}
H. Amann. Dynamic theory of quasilinear parabolic equations. II. Reaction-diffusion systems. {\it Differential Integral Equations}, 3(1):13--75, 1990.

\bibitem{Amann3}
H. Amann. Dynamic theory of quasilinear parabolic equations. III. Global existence. {\it Math. Z.}, 202(2):219--250, 1989.

\bibitem{A-Book-1993} H. Amann. Nonhomogeneous linear and quasilinear elliptic and parabolic boundary value problems. In Function spaces, differential operators and nonlinear analysis.  Teubner-Texte zur Math., Stuttgart-Leipzig, 133:9--126, 1993.


\bibitem{Barbalat1959}
I. Barb\u{a}lat. Syst\`emes d'\'{e}quations diff\'{e}rentielles
  d'oscillations non lin\'{e}aires. {\it  Rev. Math. Pures Appl.}, 4:267--270, 1959.

\bibitem{Braun} M. Braun. {\it Differential equations and theis applications}. Springer-Verlag, 1975.

\bibitem{Burkenroad}
M.D. Burkenroad. A possible function of bioluminescence. {\it J. Mar. Res}., 5:161--164, 1943.

\bibitem{CCW-AA-2020} Y. Cai, Q. Cao and Z.-A. Wang. Asymptotic dynamics and spatial patterns of a ratio-dependent predator-prey system
with prey-taxis. {\it Appl. Anal.,} 101:81-99, 2022.

\bibitem{Cantrell-IGP1}
R.S. Cantrell, X. Cao, K.Y. Lam and T. Xiang. A PDE model of intraguild predation with cross-diffusion. {\it Discrete Contin. Dyna. Syst.-B}, 22(10):3653-3661, 2017.

\bibitem{Cao-ZAMP-2016}X. Cao. Boundedness in a three-dimensional chemotaxis-haptotaxis model. {\it  Z. Angew. Math. Phys.,} 67(1):Art. 11, 13 pp, 2016.

\bibitem{Chiver}
D.P. Chivers, G.E. Brown and R.J.F. Smith. The evolution of chemical alarm signals: attracting predators benefits alarm signal senders. {\it Amer. Nat}.,\
148(4):649--659, 1996.

\bibitem{Dick}
M. Dicke, M.W. Sabelis, J. Takabayashi, J. Bruin and M.A. Posthumus. Plant strategies of manipulating predatorprey interactions through allelochemicals:
prospects for application in pest control. {\it J. Chem. Ecol}., 16(11):3091--3118, 1990.

\bibitem{Du-Shi} Y. Du and J. Shi. Some recent results on diffusive predator-prey models in spatially heterogeneous environment. Nonlinear dynamics and evolution equations, 95--135, Fields Inst. Commun., 48, Amer. Math. Soc., Providence, RI, 2006.


\bibitem{Fichtel}
C. Fichtel, S. Perry and J. Groslouis. Alarm calls of white-faced capuchin monkeys: an acoustic analysis. {\it Animal Behaviour}, 70:165--176, 2005.

\bibitem{Fuest} M. Fuest. Global solutions near homogeneous steady states in a multidimensional population model with both predator- and prey-taxis. {\it SIAM J. Math. Anal}., 52(6): 5865-5891, 2020.


\bibitem{HB-TPB-2021} E.C. Haskell and J. Bell. A model of the burglar alarm hypothesis of prey alarm calls.
{\it{Theor. Popu. Biol.}}, 141:1--13, 2021.

\bibitem{Hasting-Powell}
A. Hasting and T. Powell. Chaos in a three-species food chain. {\it Ecology}, 72(3):896-903, 1991.

\bibitem{HZ-2015} X. He and S. Zheng. Global boundedness of solutions in a reaction-diffusion system of predator-prey model with prey-taxis. {\it Appl. Math. Lett.,} 49:73--77, 2015.

\bibitem{Hogstedt}
G. Hogstedt. Adaptation unto death: function of fear screams. {\it Amer. Nat}., 121(4):562--570, 1983.

\bibitem{Holt}
R.D. Holt and G.A. Polis. A theoretical framework for intraguild predation. {\it Amer. Nat}., 149(4):745--764, 1997.

\bibitem{Hsu-survey} S.-B. Hsu. A survey of constructing Lyapunov functions for mathematical models in population biology.
{\it Taiwanese J. Math}., 9:151--173, 2005.

\bibitem{JCH-2020-Nonlinearity} C. Jin. Global solvability and stabilization to a cancer invasion model with remodelling of ECM. {\it Nonlinearity,} 33(10):5049--5079, 2020.

\bibitem{JKW-SIAP-2018} H.Y. Jin, Y.J. Kim and Z.-A. Wang. Boundedness, stabilization, and pattern formation driven by density-suppressed motility. {\it SIAM J. Appl. Math.}, 78(3):1632--1657, 2018.

\bibitem{JW-JDE-2016} H.Y. Jin and Z.-A. Wang. Boundedness, blowup and critical mass phenomenon in competing chemotaxis. {\it J. Differential Equations}, 260:162--196, 2016.

 \bibitem{JW-JDE-2017} H.Y.	Jin and Z.-A. Wang. Global stability of prey-taxis systems. {\it {J. Differential Equations,}} 262(3):1257--1290, 2017.

\bibitem{JW-EJAM-2021} H.Y. Jin and Z.-A. Wang. Global dynamics and spatio-temporal patterns of predator-prey systems with density-dependent motion. {\it  European J. Appl. Math.,} 32(4):652--682, 2021.

\bibitem{Kareiva} P. Kareiva and G. Odell.  Swarms of predators exhibit ``preytaxis" if individual predators use area-restricted search. {\it Amer. Nat.,} 130(2):233--270, 1987.

\bibitem{Klebanoff}
A. Klebanoff and A. Hastings. Chaos in three-species food chains. {\it J. Math. Biol}., 32:427--451, 1994.

\bibitem{Klump}
G. Klump and  M. Shalter. Acoustic behaviour of birds and mammals in the
predator context; I. Factors affecting the structure of alarm signals. II. The
functional significance and evolution of alarm signals. {\it Z. Tierpsychol.,} 66(3):189--226, 1984.

\bibitem{La-book-1968}
O. Lady\'{z}enskaja, V. Solonnikov and  N.N. Ural'ceva.
Translated from the Russian by S. Smith. Translations of Mathematical Monographs, Vol. 23. American Mathematical Society, Providence, R.I., 1968.


\bibitem{Li-Wang-Shao} C.L. Li, X.H. Wang and Y.F. Shao. Steady states of predator-prey system model with
prey-taxis. {\color{black} {\it Nonlinear Analysis: Theory, Methods $\&$ Applications},} 97:155--168, 2014.

\bibitem{Lieberman-1996}
G.M. Lieberman. {\em Second order parabolic differential equations}.
World Scientific Publishing Co., Inc., River Edge, NJ., 1996.


\bibitem{McCann}
K. McCann and P. Yodzis. Nonlinear dynamics and population disappearances. {\it Amer. Nat}., 144(5):873--879, 1994.

%


\bibitem{Wrzosek-mmas} P. Mishra and D. Wrzosek. The role of indirect prey-taxis and interference among predators in pattern formation. {\it Math. Methods Appl. Sci}., 43:10441--10461, 2020.

\bibitem{Mizoguchi}
N. Mizoguchi and  P.  Souplet.  Nondegeneracy of blow-up points for the parabolic Keller-Segel system. Ann. Inst. H.
Poincar\'{e} Anal. Non Lin\'{e}aire, 31:851--875, 2014.

\bibitem{Murdoch}W.W. Murdoch, J. Chesson and P.L. Chesson. Biological control in theory and practice.  {\it{Amer. Nat.,}}  125(3):344--366, 1985.

\bibitem{Patt}
D. Pattanayak, A. Mishra, S.K. Dana  and N. Bairagi. Bistability in a tri-trophic food chain model: Basin stability perspective. {\it Chaos}, 31, 073124, 2021.



\bibitem{Cantrell-IGP2}
D. Ryan and R.S. Cantrell. Avoidance behavior in intraguild predation communities: a cross-diffusion model. {\it Discrete Contin. Dyna. Syst}., 35(4):1641--1663, 2015.

\bibitem{Souplet}
 P. Souplet and P. Quittner. {\it  Superlinear Parabolic Problems: Blow-up, Global Existence and Steady States}, Birkh\"{a}user
Advanced Texts, Basel/Boston/Berlin, 2007.

\bibitem{SSW14}
C. Stinner, C. Surulescu and M. Winkler. Global weak solutions in a PDE-ODE system modeling multiscale cancer cell invasion. {\it SIAM J. Math. Anal.}, 46:1969--2007, 2014.

\bibitem{Tao-2010} Y.S. Tao. Global existence of classical solutions to a predator-prey model with nonlinear prey-taxis. {\it Nonlinear Anal. Real World Appl.,} 11(3):2056--2064, 2010.

\bibitem{Tao-Winkler-2012} Y.S. Tao and M.  Winkler. Boundedness in a quasilinear parabolic-parabolic Keller-Segel system with subcritical sensitivity. {\it J. Differential Equations}, 252:692--715, 2012.


\bibitem{Tell-Wrzosek-M3}J.I. Tello and D. Wrzosek. Predator-prey model with diffusion and indirect prey-taxis. {\it Math. Models Methods Appl. Sci}., 26:2129--2162, 2016.


\bibitem{Turchin} P. Turchin. {\color{black}{\it Complex Population Dynamics: A Theoretical/Empirical Synthesis (Monographs
in Population Biology-35)}. Princeton University Press, 2003.}

\bibitem{WW-2018-ZAMP} J.P. Wang and M.X. Wang. Boundedness and global stability of the two-predator and one-prey models with nonlinear prey-taxis. {\it {Z. Angew. Math. Phys.,}} 69(3):Paper No. 63, 24 pp, 2018.

\bibitem{WX-2021-JMB} Z.-A. Wang and J. Xu. On the Lotka-Volterra competition system with dynamical resources and density-dependent diffusion. {\it {J. Math. Biol.,}} 82(1-2):Paper No. 7, 37 pp, 2021.

\bibitem{Winkler-2010-JDE} M. Winkler. Aggregation vs. global diffusive behavior in the higher-dimensional Keller-Segel
model. {\it {J. Differential Equations,}} 248:2889--2905, 2010.

\bibitem{WSW-JDE-2016} 	S. Wu, J. Shi and B. Wu. Global existence of solutions and uniform persistence of a diffusive predator-prey model with prey-taxis. {\it {J. Differential Equations,}} 260(7):5847--5874, 2016.

\bibitem{Wu-Wang-Shi}  S. Wu, J. Wang and J. Shi. Dynamics and pattern formation of a diffusive predator-prey model with predator-taxis. {\it Math. Models Methods Appl. Sci}., 28:2275--2312, 2018.

\bibitem{Yi-Wei-Shi} F. Yi, J. Wei and J. Shi. Bifurcation and spatiotemporal patterns in a homogeneous diffusive predator–prey system. {\it J.  Differential Equations}, 246:1944--1977, 2009.


\end{thebibliography}
\end{document}